\documentclass{article}
\usepackage{lipsum}

\usepackage[english]{babel}
\usepackage{mathtools}
\usepackage{amsfonts,amsmath,amsthm,color,multirow,graphicx,amssymb}
\usepackage{epstopdf}
\usepackage[font=footnotesize]{caption}
\usepackage{subfigure}
\newcommand{\mycircle}{{\text{\small o}}}

\newcommand{\e}{\textup{e}}
\renewcommand{\i}{\textup{i}}
\newcommand{\sign}{\textup{sign}}
\newcommand{\supp}{\textup{supp}}
\newcommand{\dd}{{\rm d}}

\newcommand{\RR}{\mathbb R}
\newcommand{\CC}{\mathbb C}
\newcommand{\ZZ}{\mathbb Z}

\newcommand{\D}[4]{\prescript{#2}{#1}D_{#3}^{#4}}
\newcommand{\Dl}[1]{\D{-\infty}{}{x}{#1}}
\newcommand{\Dr}[1]{\D{x}{}{+\infty}{#1}}
\newcommand{\Ddl}[1]{\D{0}{}{x}{#1}}
\newcommand{\Ddr}[1]{\D{x}{}{1}{#1}}
\newcommand{\Ddlname}[2]{\D{0}{#2}{x}{#1}}
\newcommand{\Ddrname}[2]{\D{x}{#2}{1}{#1}}

\newcommand{\Ddlaname}[2]{\D{a}{#2}{x}{#1}}
\newcommand{\Ddrbname}[2]{\D{x}{#2}{b}{#1}}

\newcommand{\fourier}[1]{{\widehat{#1}}}
\newcommand{\conj}[1]{{\overline{#1}}}

\newtheorem{theorem}{Theorem}[section]
\newtheorem{lemma}[theorem]{Lemma}
\newtheorem{corollary}[theorem]{Corollary}
\newtheorem{remark}[theorem]{Remark}

\newtheorem{proposition}[theorem]{Proposition}
\newtheorem{definition}[theorem]{Definition}

\makeatletter
\newcommand{\mylabel}[2]{#2\def\@currentlabel{#2}\label{#1}}
\makeatother

\newcommand{\cmag}[1]{{\color{black}#1}}

\ifpdf
\DeclareGraphicsExtensions{.eps,.pdf,.png,.jpg}
\else
\DeclareGraphicsExtensions{.eps}
\fi

\date{}
\title{On the matrices in B-spline collocation methods for Riesz fractional equations and their spectral properties}

\author{M.\ Mazza$^{1,}$\thanks{Corresponding author
		\newline
		${}^1$ Department of Humanities and Innovation, University of Insubria, via Valleggio 11, 22100 Como, Italy (\texttt{mariarosa.mazza@uninsubria.it})
		\newline
		${}^2$ Department of Science and High Technology, University of Insubria, via Valleggio 11, 22100 Como, Italy (\texttt{marco.donatelli@uninsubria.it})
		\newline
		${}^3$ Department of Mathematics, University of Rome Tor Vergata, Via della Ricerca Scientifica 1, 00133 Rome, Italy (\texttt{\{manni,speleers\}@mat.uniroma2.it})
}\and M.\ Donatelli$^{2}$\and C.\ Manni$^{3}$\and H.\ Speleers$^{3}$}

\usepackage{amsopn}

\begin{document}

\maketitle

\begin{abstract}
In this work, we focus on a fractional differential equation in Riesz form discretized by a polynomial B-spline collocation method. For an arbitrary polynomial degree $p$, we show that the resulting coefficient matrices possess a Toeplitz-like structure. We investigate their spectral properties via their symbol and we prove that, like for second order differential problems, also in this case the given matrices are ill-conditioned both in the low and high frequencies for large $p$. More precisely, in the fractional scenario the symbol has a single zero at $0$ of order $\alpha$, with $\alpha$ the fractional derivative order that ranges from $1$ to $2$, and it presents an exponential decay to zero at $\pi$ for increasing $p$ that becomes faster as $\alpha$ approaches $1$. This translates in a mitigated conditioning in the low frequencies and in a deterioration in the high frequencies when compared to second order problems. Furthermore, the derivation of the symbol reveals another similarity of our problem with a classical diffusion problem. Since the entries of the coefficient matrices are defined as evaluations of fractional derivatives of the B-spline basis at the collocation points, we are able to express the central entries of the coefficient matrix as inner products of two fractional derivatives of cardinal B-splines. Finally, we perform a numerical study of the approximation behavior of polynomial B-spline collocation. This study suggests that, in line with non-fractional diffusion problems, the approximation order \cmag{for smooth solutions} in the fractional case is $p+2-\alpha$ for even $p$, and $p+1-\alpha$ for odd $p$.

\end{abstract}



\section{Introduction}

Fractional diffusion equations (FDEs) generalize classical partial differential equations (PDEs). Their recent success is due to the non-local behavior of fractional operators resulting in an appropriate modeling of anomalous diffusion phenomena that appear in several applicative fields, like imaging or electrophysiology \cite{imag,orovio}. In particular, a standard diffusion equation can be ``fractionalized'', either by replacing the derivative in time with a fractional one whose fractional order ranges from $0$ to $1$, or by introducing a fractional derivative in space with order between $1$ and $2$. The two approaches can also be combined and lead to similar computational issues. 

\begin{sloppypar}
The improved physical description of the considered phenomenon obtained by ``fractionalizing'' the derivatives, however, translates in a \cmag{more challenging} numerical treatment of the corresponding discretized problems. Indeed, \cmag{the evaluation/approximation of a fractional operator is numerically more expensive (and often less stable). Moreover,} even when standard local discretization methods are adopted, the non-locality of the fractional operators causes absence of sparsity in the discretization matrices. This makes FDEs computationally more demanding than PDEs. 
\end{sloppypar}

\cmag{Various numerical discretization methods for FDE problems (e.g., finite differences, finite volumes, finite elements, spectral methods) can be found in the literature. We refer the reader to
	\cite{ervin,Moroney,LW,Mao-sinum-18,Meer1,TZD,WD,Mao-sisc-17} and references therein.
	In the case of regular spatial domain subdivisions,} the discretization matrices inherit a Toeplitz-like structure from the space-invariant property of the underlying operators that can be exploited for the design of ad hoc iterative schemes of multigrid and preconditioned Krylov type (see, e.g., \cite{mazza2016,sisc,LS,mazza2017,PNW,PS}).
In the context of finite difference/volume discretizations, we mention the structure preserving preconditioning and the algebraic multigrid methods presented in \cite{mazza2016,sisc}. Both strategies are based on the spectral analysis of the coefficient matrices via their symbol, a function which provides an approximation of their eigenvalues/singular values.

\cmag{A similar symbol-based approach has also been successfully employed in the context of isogeometric analysis (IgA) for the discretization of integer order differential problems;} see, e.g., \cite{cmame2,sinum,MMRSS}. In these papers, the spectral information provided by the symbol has been leveraged for the design of effective preconditioners and fast multigrid/multi-iterative solvers whose convergence speed is independent of the fineness parameters and the approximation parameters.

\cmag{The present work aims at uncovering the structure and studying the symbol of the discretization matrices obtained by IgA collocation for FDE problems. As a first step towards the spectral treatment of general differential problems involving fractional diffusion operators, we consider here} the following fractional diffusion boundary value problem with absorbing boundary conditions:
\begin{equation}\label{eq:FDE}
\begin{cases}
\frac{\dd^\alpha u(x)}{\dd|x|^\alpha}=s(x), & x\in\Omega,\\
u(x)=0, & x\in\RR\backslash\Omega,
\end{cases}
\end{equation}
where $\Omega:=(0,1)$, $\alpha\in(1,2)$, and 
\[\frac{\dd^\alpha u(x)}{\dd|x|^\alpha}:=\frac{1}{2\cos({\pi\alpha}/{2})}\left(\Ddlname{\alpha}{RL}u(x)+\Ddrname{\alpha}{RL}u(x)\right)\]
is the so-called Riesz fractional operator, while $\Ddlname{\alpha}{RL}u(x)$, $\Ddrname{\alpha}{RL}u(x)$ are the left and right Riemann-Liouville fractional derivatives of $u$ (see Section~\ref{sub:fracder} for their definition). 
More precisely, we are interested in a polynomial B-spline collocation-based discretization of \eqref{eq:FDE} where the so-called Greville abscissae are chosen as collocation points. 

Collocation methods based on polynomial splines were applied to fractional problems for the first time in \cite{blank} and further developed in \cite{pedas}. 
Polynomial B-spline bases have been used for solving time-fractional problems in \cite{Pitolli} and (left-sided) space-fractional problems in \cite{Pitolli2}.
Among non-polynomial spline collocation methods for fractional problems, we mention the work \cite{PP} in which the authors explore the application of fractional B-splines.

Our choice of classical polynomial B-splines is motivated by the fact that, contrarily to their fractional counterpart, they have compact support and naturally fulfill boundary and/or initial conditions. Furthermore, they possess good approximation properties. Seminal results concerning the structure of the quadratic spline collocation matrices can be found in \cite{QSC2}. Therein, the authors recognize the Toeplitz-like structure of the coefficient matrices and use a classical circulant preconditioner to solve the corresponding linear systems by means of Krylov methods. 

To the best of our knowledge, this is the first time that the structure and the spectral properties of polynomial B-spline collocation matrices are investigated for an arbitrary polynomial degree $p$. We show that the coefficient matrices retain the Toeplitz-like structure and we study their spectral properties via their symbol. It turns out that the symbol: 
\begin{enumerate}
	\item[{(a)}] has a single zero at $0$ of order $\alpha$; 
	\item[{(b)}] presents an exponential decay to zero at $\pi$ for increasing $p$, a so-called numerical zero, that becomes faster as $\alpha$ approaches $1$;
	\item[{(c)}] is bounded in the proximity of $\pi$.
\end{enumerate}
This translates in a mitigated conditioning in the low frequencies and in a deterioration in the high frequencies when compared to second order problems (see \cite{DGMSS}). \cmag{The symbol, and so the (asymptotic) spectral properties of the involved matrices, do not change if reaction and/or advection terms are added to \eqref{eq:FDE}.
	
	As a side result of the symbol computation,} we propose a new way of expressing both a left and a right fractional derivative of a cardinal B-spline as inner products of two fractional derivatives of cardinal B-splines. 

Furthermore, we provide a numerical study of the approximation behavior of polynomial B-spline collocation for an arbitrary degree $p$. \cmag{It turns out that the approximation order for smooth solutions is $p+2-\alpha$ for even $p$, and $p+1-\alpha$ for odd $p$. This is again in agreement with the approximation results known for standard (non-fractional) diffusion problems \cite{ABHRS}.
	We refer the reader to \cite{Cai-2019,Hao-sinum-20,Mao-sinum-18} for a smoothness analysis of the solution in (weighted) Sobolev spaces.}


The paper is organized as follows. Section~\ref{sec:preliminaries} is devoted to notations, definitions, and preliminary results. In Section~\ref{sec:frac-cardinal} we present a new way of writing the fractional derivative of a cardinal B-spline. 
In Section~\ref{sec:iga} we describe the IgA collocation approximation of the problem reported in \eqref{eq:FDE}, while in Sections~\ref{sec:symbol-properties} and \ref{sec:symbol} we perform a detailed spectral analysis of the resulting coefficient matrices. We validate our theoretical spectral findings with a selection of numerical experiments in Section~\ref{sec:numerics} and we do a numerical study of the approximation order of the polynomial B-spline collocation method as well. We end with some concluding remarks in Section~\ref{sec:conclusions}.

\section{Preliminaries}\label{sec:preliminaries}

In this section we collect some preliminary tools on fractional derivatives, spectral analysis and IgA discretizations. Firstly, we give two definitions of fractional derivatives (Section~\ref{sub:fracder}). Secondly, after introducing
the definition of spectral distribution of general matrix-sequences, we summarize the essentials of Toeplitz sequences (Section~\ref{sub:tools}). Finally, we recall the definition of B-splines and cardinal B-splines (Section~\ref{sub:bsplines}).

\subsection{Fractional derivatives}\label{sub:fracder}
A common definition of fractional derivatives is given by the Riemann-Liouville formula. For a given function $u$ with absolutely continuous $(m-1)$-th derivative on $[a,b]$, the left and right Riemann-Liouville fractional derivatives of order $\alpha$ are defined by
\begin{align*}
\begin{split}
\Ddlaname{\alpha}{RL}u(x)&:=\frac{1}{\Gamma(m-\alpha)}\frac{{\dd}^m}{\dd x^m}\int_{a}^x(x-y)^{m-\alpha-1}u(y)\,\dd y,\\
\Ddrbname{\alpha}{RL}u(x)&:=\frac{(-1)^m}{\Gamma(m-\alpha)}\frac{{\dd}^m}{\dd x^m}\int_{x}^{b}(y-x)^{m-\alpha-1}u(y)\,\dd y,
\end{split}
\end{align*}
with $m$ the integer such that $m-1\le\alpha<m$ and $\Gamma$ the Euler gamma function. Note that the left fractional derivative of the function $u$ computed at $x$ depends on all function values to the left of $x$, while the right fractional derivative depends on the ones to the right.

Another common definition of fractional derivative was proposed by Caputo:
\begin{align}\label{eq:caputo}
\begin{split}
\Ddlaname{\alpha}{C}u(x)&:=\frac{1}{\Gamma(m-\alpha)}\int_{a}^x(x-y)^{m-\alpha-1}u^{(m)}(y)\,\dd y,\\
\Ddrbname{\alpha}{C}u(x)&:=\frac{(-1)^m}{\Gamma(m-\alpha)}\int_{x}^{b}(y-x)^{m-\alpha-1}u^{(m)}(y)\,\dd y.
\end{split}
\end{align}
Note that \eqref{eq:caputo} requires the $m$-th derivative of $u$ to be absolutely integrable. Higher regularity of the solution is typically imposed in time rather than in space. As a consequence, the Caputo formulation is mainly used for fractional derivatives in time, while Riemann-Liouville's is preferred for fractional derivatives in space. The use of Caputo's derivative provides some advantages in the treatment of boundary conditions when applying the Laplace transform method (see \cite[Chapter~2.8]{Podlubny}).

The Riemann-Liouville derivatives are related to the Caputo ones as follows
\begin{align}\label{eq:relRLC}
\begin{split}
\Ddlaname{\alpha}{RL}u(x)&=\Ddlaname{\alpha}{C}u(x)+ \sum_{k=0}^{m-1} \frac{(x-a)^{k-\alpha}}{\Gamma(k-\alpha+1)}u^{(k)}(a^+),\\
\Ddrbname{\alpha}{RL}u(x)&=\Ddrbname{\alpha}{C}u(x)+ \sum_{k=0}^{m-1} \frac{(-1)^k(b-x)^{k-\alpha}}{\Gamma(k-\alpha+1)}u^{(k)}(b^-),
\end{split}
\end{align}
and the two coincide if $u$ satisfies homogeneous conditions, i.e., $u^{(k)}(a^+)=u^{(k)}(b^-)=0$ for $k=0,\ldots,m-1$.

\begin{remark}
	\cmag{Throughout the paper, whenever we write $\Ddlaname{\alpha}{RL}u(\xi)$ or $\Ddrbname{\alpha}{RL}u(\xi)$ for a fixed $\xi$ we mean
		$\Ddlaname{\alpha}{RL}u(x)$ or $\Ddrbname{\alpha}{RL}u(x)$, where $x=\xi$,
		respectively.}
\end{remark}

\subsection{Spectral tools}\label{sub:tools}
We begin with the formal definition of spectral distribution in the sense of the eigenvalues for a general matrix-sequence.

\begin{definition}
	Let $f:G\to\CC$ be a measurable function, defined on a measurable set $G\subset\RR^k$ with $k\ge 1$ and Lebesgue measure $0<\mu_k(G)<\infty$. Let  $\mathcal C_0(\CC)$ be the set of continuous functions with compact support over $\CC$, and let $A_n$ be a matrix of size $d_n$ with eigenvalues $\lambda_j(A_n)$, $j=1,\ldots,d_n$.
	The matrix-sequence $\{A_n\}_n$ (with $d_n<d_{n+1}$) is {distributed as the pair $(f,G)$ in the sense of the eigenvalues,} denoted by $$\{A_n\}_n\sim_\lambda(f,G),$$ if the following limit relation holds for all $F\in\mathcal C_0(\CC)$:
	\begin{align}\label{distribution:sv-eig}
	\lim_{n\to\infty}\frac{1}{d_n}\sum_{j=1}^{d_n}F(\lambda_j(A_n))=
	\frac1{\mu_k(G)}\int_G F(f(t))\, \dd t.
	\end{align}
	We say that $f$ is the {(spectral) symbol} of the matrix-sequence $\{A_{n}\}_{n}$.
\end{definition}

\begin{remark}
	Throughout the paper, when it is not of crucial importance to know what is the domain of $f$, we replace the notation $\{A_n\}_n\sim_\lambda(f,G)$ with $\{A_n\}_n\sim_\lambda f$.
\end{remark}

\begin{remark}
	When $f$ is continuous, an informal interpretation of the limit relation \eqref{distribution:sv-eig} is that when the matrix-size is sufficiently large, the eigenvalues of $A_n$ can be approximated by a sampling of $f$ on a uniform equispaced grid of the domain $G$.
\end{remark}

The following result allows us to determine the spectral distribution of a Hermitian matrix-sequence plus a correction (see \cite{BarSer}).

\begin{theorem}\label{thm:quasi-herm}
	Let $\{X_n\}_n$ and $\{Y_n\}_n$ be two matrix-sequences, with $X_n,Y_n\in\CC^{d_n\times d_n}$, and assume that 
	\begin{itemize}
		\item[{(a)}] $X_n$ is Hermitian for all $n$ and $\{X_n\}_n\sim_{\lambda}f$;
		\item[{(b)}] $\|Y_n\|_F=o(\sqrt{d_n})$ as $n\rightarrow\infty$, with $\|\cdot\|_{F}$ the Frobenius norm.
	\end{itemize}
	Then, $\{X_n+Y_n\}_n\sim_{\lambda}f$.
\end{theorem}

For a given matrix $X\in\CC^{m\times m}$, let us denote by $\|X\|_{1,\ast}$ the trace norm defined by $\|X\|_{1,\ast}:=\sum_{j=1}^{m}\sigma_j(X)$, where $\sigma_j(X)$ are the $m$ singular values of $X$.

\begin{corollary}\label{cor:quasi-herm}
	Let $\{X_n\}_n$ and $\{Y_n\}_n$ be two matrix-sequences, with $X_n,Y_n\in\CC^{d_n\times d_n}$, and assume that (a) in Theorem~\ref{thm:quasi-herm} is satisfied. Moreover, assume that any of the following two conditions is met:
	\begin{itemize}
		\item $\|Y_n\|_{1,\ast}=o(\sqrt{d_n})$;
		\item $\|Y_n\|_2=o(1)$, with $\|\cdot\|_2$ the spectral norm.
	\end{itemize}
	Then, $\{X_n+Y_n\}_n\sim_{\lambda}f$.	
\end{corollary}

\begin{remark}
	In \cite{BarSer} the authors conjecture that Theorem~\ref{thm:quasi-herm} is also valid when (b) is replaced by the weaker condition $\|Y_n\|_{1,\ast}=o(d_n)$.
\end{remark}

We now recall the definition of Toeplitz sequences generated by univariate functions in $L^1([-\pi, \pi])$.
\begin{definition}
	Let $f \in L^1([-\pi, \pi])$ and let $ f_k$ be its Fourier coefficients,
	\begin{equation}\label{eq:toeplitz-coeff}
	f_k:=\frac{1}{2\pi}\int_{-\pi}^{\pi} f(\theta)\e^{-\i(k\theta)}\,\dd\theta,\quad k\in\ZZ.
	\end{equation}
	The $n$-th Toeplitz matrix associated with $f$ is the $n \times n$ matrix defined by
	\begin{equation}\label{eq:toeplitz}
	T_{n}(f):=\begin{bmatrix}
	f_0 & f_{-1} & \cdots & \cdots & f_{-(n-1)} \\
	f_1 & \ddots & \ddots & & \vdots \\
	\vdots & \ddots & \ddots & \ddots & \vdots\\
	\vdots & & \ddots & \ddots & f_{-1}\\
	f_{n-1} & \cdots & \cdots & f_1 & f_0
	\end{bmatrix}\in\CC^{n\times n}.
	\end{equation}
	The matrix-sequence $\{T_n(f)\}_{n}$ is called the {Toeplitz sequence generated by $f$}.
\end{definition}

For real-valued Toeplitz matrix-sequences, the following theorem holds (see, e.g., \cite{GSz}).
\begin{theorem}\label{szego}
	Let $f\in L^1([\pi,\pi])$ be a real-valued function. Then, $$\{T_n(f)\}_n\sim_\lambda(f,[-\pi,\pi]).$$
\end{theorem}

\subsection{B-splines and cardinal B-splines}\label{sub:bsplines}
For $p\ge0$ and $n\ge1$, consider the following uniform knot sequence
\begin{equation*}
\xi_1=\cdots=\xi_{p+1}:=0<\xi_{p+2}<\cdots<\xi_{p+n}<1=:\xi_{p+n+1}=\cdots=\xi_{2p+n+1},
\end{equation*}
where
\begin{equation*}
\xi_{i+p+1}:=\frac{i}{n}, \quad i=0,\ldots,n.
\end{equation*}
This knot sequence allows us to define $n+p$ B-splines of degree $p$.
\begin{definition}
	The B-splines of degree $p$ over a uniform mesh of $[0,1]$, consisting of $n$ intervals, are denoted by
	\begin{equation*}
	N_{i}^p:[0,1]\rightarrow \RR, \quad i=1,\ldots,n+p,
	\end{equation*}
	and defined recursively as follows: for $1 \le i\le n+2p$,
	\begin{equation*}
	N_i^0(x):=\begin{cases}
	1, & x \in [\xi_i,\xi_{i+1}), \\
	0, & \text{otherwise};
	\end{cases}
	\end{equation*}
	for $1\le k\le p$ and $1\le i\le n+2p-k$,
	\begin{equation*}
	N_i^k(x):=\frac{x-\xi_i}{\xi_{i+k}-\xi_i}N_{i}^{k-1}(x)+\frac{\xi_{i+k+1}-x}{\xi_{i+k+1}-\xi_{i+1}}N_{i+1}^{k-1}(x),
	\end{equation*}
	where a fraction with zero denominator is assumed to be zero.
\end{definition}
It is well known that the B-splines $N_{i}^p$, $i=1,\ldots,n+p$, are linearly independent and they enjoy the following list of properties (see, e.g.,~\cite{deBoor,LycheMS2018}).
\begin{itemize}
	\item Local support:
	\begin{equation}\label{eq:spline-support}
	\supp(N_{i}^p)=[\xi_i,\xi_{i+p+1}], \quad i=1,\ldots,n+p;
	\end{equation}
	\item Smoothness: 
	\begin{equation*}
	N_{i}^p \in \mathcal{C}^{p-1}(0,1), \quad i=1,\ldots,n+p;
	\end{equation*}
	\item Differentiation: 
	\begin{equation}\label{eq:spline-diff}
	\left(N_{i}^p(x)\right)' = p\left(\frac{N_{i}^{p-1}(x)}{\xi_{i+p}-\xi_i}-
	\frac{N_{i+1}^{p-1}(x)}{\xi_{i+p+1}-\xi_{i+1}}\right), \quad i=1,\ldots,n+p, \quad p \geq 1;
	\end{equation}
	\item Non-negative partition of unity:
	\begin{equation*}
	N_{i}^p(x)\ge0, \quad i=1,\ldots,n+p, \qquad \sum_{i=1}^{n+p}N_{i}^p(x)=1;
	\end{equation*}
	\item Vanishing at the boundary:  
	\begin{equation} \label{eq:spline-boundary}
	N_{i}^p(0)=N_{i}^p(1)=0, \quad i=2,\ldots,n+p-1;
	\end{equation}
	\item Bound for the second derivatives:
	\begin{equation} \label{eq:spline-diff-bound}
	|(N_{i}^p(x))''|\le 4p(p-1)n^2, \quad x\in(0,1).
	\end{equation}
\end{itemize}
We also add a property concerning fractional derivatives, which follows from \eqref{eq:relRLC} and \eqref{eq:spline-diff}--\eqref{eq:spline-boundary}.
\begin{itemize}
	\item The Riemann-Liouville and the Caputo derivatives of interior B-splines coincide:
	\begin{equation}\label{eq:spline-frac}
	\begin{array}{c}
	\Ddlname{\alpha}{RL}N^p_{i}=\Ddlname{\alpha}{C}N^p_{i}\\
	\Ddrname{\alpha}{RL}N^p_{i}=\Ddrname{\alpha}{C}N^p_{i}
	\end{array},
	\quad i=m+1,\ldots,n+p-m.
	\end{equation}
\end{itemize}

From now onwards, we will denote the left and right Riemann-Liouville derivatives 
simply by $\Ddl{\alpha}$ and $\Ddr{\alpha}$. In view of the last B-spline property, these also stand for the left and right Caputo derivatives in case of interior B-splines.

The B-splines $N_i^p$, $i=p+1,\ldots,n$, are uniformly shifted and scaled versions of a single shape function, the so-called cardinal B-spline $\phi_p:\RR\rightarrow \RR$,
\begin{equation}\label{eq:phi_0}
\phi_0(t) := \begin{cases}
1, & t \in [0, 1), \\
0, & \text{otherwise},
\end{cases}
\end{equation}
and
\begin{equation}\label{eq:phi_recurrence}
\phi_p (t) := \frac{t}{p} \phi_{p-1}(t) + \frac{p+1-t}{p} \phi_{p-1}(t-1), \quad p \geq 1.
\end{equation}
More precisely, we have
\begin{equation*}
N^{p}_i(x) =\phi_{p}(nx-i+p+1), \quad i=p+1,\ldots,n,
\end{equation*}
and
\begin{equation*}
\left(N^{p}_i(x)\right)' =n\phi'_{p}(nx-i+p+1), \quad i=p+1,\ldots,n.
\end{equation*}
The cardinal B-spline $\phi_p$ belongs to $\mathcal{C}^{p-1}(\RR)$ and is supported on the interval $[0,p+1]$. It is a symmetric function with respect to $\frac{p+1}{2}$, the midpoint of its support.
The left Caputo derivative of $\phi_p$ has the following explicit expression (see \cite{Pitolli}):
\begin{equation}\label{eq:cardinal_central}
\D{0}{}{t}{\alpha}\phi_p(t)=\frac{1}{\Gamma(p-\alpha+1)}\sum_{j=0}^{p+1}(-1)^j\binom{p+1}{j}(t-j)^{p-\alpha}_+, \quad 0\leq\alpha<p,
\end{equation}
where $(\cdot)^q_+$ is the truncated power function of degree $q$. Note that the function in \eqref{eq:cardinal_central} is a fractional spline, i.e., a spline with fractional degree \cite{siamREV}.
For other common properties of cardinal B-splines, we refer the reader to \cite[Section~3.1]{GMPSS}.

\section{Fractional derivatives of cardinal B-splines} \label{sec:frac-cardinal}

The aim of this section is to write the fractional derivative of a cardinal B-spline as the inner product of fractional derivatives of cardinal B-splines (see Theorem~\ref{thm:inner-product-mix}). This result will be used in Section~\ref{sec:symbol-properties} to derive an explicit expression of the symbol of the coefficient matrices of interest.

All the results in this section refer to fractional derivatives on the half-axes. More precisely, for a given compactly supported function $u$ with absolutely continuous $(m-1)$-th derivative on $\RR$, we consider
\begin{align}\label{eq:fractional-infty}
\begin{split}
\Dl{\alpha}u(x)&:=\frac{1}{\Gamma(m-\alpha)}\frac{{\dd}^m}{\dd x^m}\int_{-\infty}^x(x-y)^{m-\alpha-1}u(y)\,\dd y,\\
\Dr{\alpha}u(x)&:=\frac{(-1)^m}{\Gamma(m-\alpha)}\frac{{\dd}^m}{\dd x^m}\int_{x}^{+\infty}(y-x)^{m-\alpha-1}u(y)\,\dd y,
\end{split}
\end{align}
with $m$ the integer such that $m-1\leq\alpha<m$. For functions $u$ that are solutions of problem \eqref{eq:FDE} and $m=2$, these derivatives reduce to $\Ddl{\alpha}u(x)$ and $\Ddr{\alpha}u(x)$ since the adopted boundary conditions ensure $u$ to be identically zero on $\RR\backslash(0,1)$.

Let $\fourier{f}$ denote the Fourier transform of $f\in L_2(\RR)$, i.e.,
\begin{equation*}
\widehat{f}(\theta):=\int_\RR f(x) \e^{-\i\,\theta x}\,\dd x.
\end{equation*}
We start with a lemma addressing the Fourier transform of the derivatives in \eqref{eq:fractional-infty} for cardinal B-splines.

\begin{lemma}\label{lem:fractional-fourier}
	Let $\phi_{p}$ be the cardinal B-spline as defined in (\ref{eq:phi_0})--(\ref{eq:phi_recurrence}).  Then, for $0\leq\alpha<p$ we have
	\begin{align} \label{eq:fractional-l-fourier}
	\fourier{{\Dl{\alpha}\phi_{p}}}(\theta)=(\i\theta)^\alpha\left(\frac{1-\e^{-\i\theta}}{\i\theta}\right)^{p+1},
	\end{align}
	and
	\begin{align} \label{eq:fractional-r-fourier}
	\fourier{{\Dr{\alpha}\phi_{p}}}(\theta)=(-\i\theta)^\alpha\left(\frac{1-\e^{-\i\theta}}{\i\theta}\right)^{p+1}.
	\end{align}
\end{lemma}
\begin{proof}
	From \cite{Podlubny} we know that
	\begin{equation*}
	\fourier{{\Dl{\alpha}f}}(\theta)=(\i\theta)^\alpha\fourier{f}(\theta), \quad
	\fourier{{\Dr{\alpha}f}}(\theta)=(-\i\theta)^\alpha\fourier{f}(\theta),
	\end{equation*}
	and from \cite{Chui,LycheMS2018},
	\begin{equation*}
	\fourier{\phi_{p}}(\theta)
	=\left(\frac{1-\e^{-\i\theta}}{\i\theta}\right)^{p+1}.
	\end{equation*}
	Combining these results immediately gives \eqref{eq:fractional-l-fourier} and \eqref{eq:fractional-r-fourier}.
\end{proof}
In the following $\alpha_1,\alpha_2$ stand for real numbers.

\begin{lemma}\label{lem:im-power-mix}
	Let $\conj{z}$ denote the conjugate of the complex number $z$. Then, for any real number $\theta$ we have
	\begin{equation*}
	(\i\theta)^{\alpha_1}\conj{(-\i\theta)^{\alpha_2}}=(\i\theta)^{\alpha_1+\alpha_2}.
	\end{equation*}
\end{lemma}
\begin{proof}
	Let us consider the polar form of the complex number $(\i\theta)^\alpha$, i.e.,
	\begin{equation*}
	(\i\theta)^\alpha=|\theta|^\alpha\e^{\i\frac{\pi}{2}\sign(\theta)\alpha}.
	\end{equation*}
	Hence,
	\begin{align*}
	(\i\theta)^{\alpha_1}\conj{(-\i\theta)^{\alpha_2}}
	&= |\theta|^{\alpha_1}\e^{\i\frac{\pi}{2}\sign(\theta)\alpha_1}|\theta|^{\alpha_2}\e^{\i\frac{\pi}{2}\sign(\theta)\alpha_2}
	=|\theta|^{\alpha_1+\alpha_2}\e^{\i\frac{\pi}{2}\sign(\theta)(\alpha_1+\alpha_2)},
	\end{align*}
	which completes the proof.
\end{proof}

We are now ready for the main result of this section. 
\begin{theorem} \label{thm:inner-product-mix}
	Let $\phi_{p}$ be the cardinal B-spline as defined in (\ref{eq:phi_0})--(\ref{eq:phi_recurrence}). Then, for $0\leq\alpha_1<p_1$ and $0\leq\alpha_2<p_2$ we have
	\begin{align} 
	\int_\RR \Dl{\alpha_1}\phi_{p_1}(x)\,{\Dr{\alpha_2}\phi_{p_2}(x+k)}\,\dd x
	&= \Dl{\alpha_1+\alpha_2}\phi_{p_1+p_2+1}(p_2+1-k), 
	\label{eq:inner-product-mix} \\
	\int_\RR \Dr{\alpha_1}\phi_{p_1}(x)\,{\Dl{\alpha_2}\phi_{p_2}(x+k)}\,\dd x
	&= \Dr{\alpha_1+\alpha_2}\phi_{p_1+p_2+1}(p_2+1-k). 
	\label{eq:inner-product-mix2}
	\end{align}
\end{theorem}
\begin{proof}
	We first recall the Parseval identity for Fourier transforms, i.e.,
	\begin{equation*} 
	\int_\RR \varphi(x)\conj{\psi(x)}\,\dd x = \frac{1}{2\pi} \int_\RR \fourier{\varphi}(\theta)\conj{\fourier{\psi}(\theta)}\,\dd \theta,\quad\varphi,\,\psi\in L_2(\RR),
	\end{equation*}
	and the translation property of the Fourier transform, i.e.,
	\begin{equation*} 
	\fourier{\psi(\cdot+k)}(\theta) = \fourier{\psi(\theta)}\,\e^{\i k\theta},\quad\psi\in L_1(\RR),\ k\in\RR.
	\end{equation*}
	Starting from the above equalities, and using Lemmas~\ref{lem:fractional-fourier} and \ref{lem:im-power-mix}, we get
	\begin{align*}
	&\int_\RR \Dl{\alpha_1}\phi_{p_1}(x)\,{\Dr{\alpha_2}\phi_{p_2}(x+k)}\,\dd x \\
	&\quad=\frac{1}{2\pi}\int_\RR \fourier{\Dl{\alpha_1} \phi_{p_1}}(\theta)\,\conj{\fourier{\Dr{\alpha_2}\phi_{p_2}}(\theta)\e^{\i k\theta}}\,\dd \theta \\
	&\quad=\frac{1}{2\pi}\int_\RR (\i\theta)^{\alpha_1+\alpha_2}\left(\frac{1-\e^{-\i\theta}}{\i\theta}\right)^{p_1+1}\left(\frac{\e^{\i\theta}-1}{\i\theta}\right)^{p_2+1}\e^{-\i k\theta}\,\dd \theta \\
	&\quad=\frac{1}{2\pi}\int_\RR (\i\theta)^{\alpha_1+\alpha_2}\left(\frac{1-\e^{-\i\theta}}{\i\theta}\right)^{p_1+p_2+2}\e^{\i(p_2+1-k)\theta}\,\dd \theta.
	\end{align*}
	By taking the inverse Fourier transform of the right-hand side we arrive at~\eqref{eq:inner-product-mix}. The proof of \eqref{eq:inner-product-mix2} is analogous.
\end{proof}

\begin{remark}
	Theorem~\ref{thm:inner-product-mix} is a generalization of a known explicit formula for inner products of integer derivatives of cardinal B-splines (see \cite{GMPSS,LycheMS2018} and also \cite{Speleers2015}).
\end{remark}

\section{IgA collocation discretization of the fractional Riesz operator} \label{sec:iga}
From now onwards, we assume that $\alpha$ is fixed in the open interval $(1,2)$.
Let $\mathcal{W}$ be a finite dimensional vector space of sufficiently smooth functions defined on the closure of $\Omega$ and vanishing at its boundary, and let $N:={\rm dim}(\mathcal{W})$. Applying the collocation method to \eqref{eq:FDE} means looking for a function $u_{\mathcal{W}}\in \mathcal{W}$ such that
\begin{equation}\label{eq:colloc-system}
\frac{\dd^\alpha u_{\mathcal{W}}(x_i)}{\dd|x|^\alpha}=s(x_i), \quad i=1,\ldots,N,
\end{equation}
with $x_i\in\Omega$, the so-called collocation points. Given a basis $\{\varphi_j : j=1,\ldots,N\}$ of $\mathcal{W}$, problem \eqref{eq:colloc-system} can be rewritten in matrix form as follows:
\begin{equation*}
A_{\rm col}{\bf u}={\bf b}_{\rm col},
\end{equation*}
with 
\begin{equation*}
A_{\rm col}:=\left[\frac{1}{2\cos({\pi\alpha}/{2})}(\Ddl{\alpha}+\Ddr{\alpha})\varphi_j(x_i)\right]_{i,j=1}^{N}, \quad {\bf b}_{\rm col}:=[s(x_i)]_{i=1}^N,
\end{equation*}
and ${\bf u}:=[u_1,\ldots,u_N]^T$ such that  $u_{\mathcal{W}}(x)=\sum_{j=1}^{N}u_j\varphi_j(x)$. In this paper, we choose $\mathcal{W}$ as the space of splines of degree $p\geq2$ that vanish at the boundary, and the collocation points as the Greville abscissae. More precisely, we take
\begin{itemize}
	\item the approximation space as the space spanned by the B-splines of degree $p\geq2$ that are zero at the boundary (see \eqref{eq:spline-boundary}), i.e., 
	\begin{equation}\label{eq:spline-space}
	\mathbb{S}^{p}_{n}:={\rm span}\{N^p_i : i=2,\ldots,n+p-1\};
	\end{equation}
	\item the collocation points as the Greville abscissae corresponding to the B-splines in \eqref{eq:spline-space}, i.e.,
	\begin{equation*}
	\eta_i:=\frac{\xi_{i+1}+\cdots+\xi_{i+p}}{p}, \quad i=2,\ldots,n+p-1.
	\end{equation*}	
\end{itemize}
Thus \eqref{eq:colloc-system} translates in the following linear system
\begin{equation*}
A^{p,\alpha}_n{\bf u}_n={\bf b}_n,
\end{equation*}
where 
\begin{equation*}
A^{p,\alpha}_n:=\frac{1}{2\cos({\pi\alpha}/{2})}(A_n^L+A_n^R), \quad {\bf b}_n:=[s(\eta_{i+1})]_{i=1}^{n+p-2},
\end{equation*}
with
\begin{equation}\label{eq:ALR}
A_n^{L}:=\left[\Ddl{\alpha}N^p_{j+1}(\eta_{i+1})\right]_{i,j=1}^{n+p-2}, \quad A_n^R:=\left[\Ddr{\alpha}N^p_{j+1}(\eta_{i+1})\right]_{i,j=1}^{n+p-2},
\end{equation}
and ${\bf u}_n:=[u_1,\dots,u_{n+p-2}]^T$, the vector of the coefficients of $u$ with respect to the B-spline basis functions in the space $\mathbb{S}_n^p$.

In order to assemble the matrices $A_n^L$ and $A_n^R$, we need to compute the left and right fractional derivatives of any B-spline. By using \eqref{eq:cardinal_central}, for the B-splines $N_{i}^p$ corresponding to the indexes $i=p+1,\ldots,n$, we have
\begin{align*}
\Ddl{\alpha}N_i^p(x)&=n^\alpha \D{0}{}{nx}{\alpha}\phi_p(nx-i+p+1)\\
&=\frac{n^\alpha}{\Gamma(p-\alpha+1)}\sum_{j=0}^{p+1}(-1)^j\binom{p+1}{j}\left(nx-i+p+1-j\right)^{p-\alpha}_+.
\end{align*}
Thanks to this relation, and recalling that the Greville abscissae for $i=p+1,\ldots,n$ reduce to
\begin{equation*}
\eta_i=\frac{i}{n}-\frac{p+1}{2n}, \quad i=p+1,\ldots,n,
\end{equation*}
or equivalently,
\begin{equation*}
n\eta_i+p+1=i+\frac{p+1}{2}, \quad i=p+1,\ldots,n,
\end{equation*}
we can immediately recognize that the central part of the matrix $A_n^L$ corresponding to the indexes $p+1,\ldots,n$ has a Toeplitz structure. In other words, we have
\[A_n^L=n^\alpha (T^L_n+R^L_n),\]
where
\[T^L_n:=\left[ \D{0}{}{nx}{\alpha}\phi_p\left(\frac{p+1}{2}+i-j\right)\right]_{i,j=1}^{n+p-2},\]
and $R^L_n$ is a matrix whose rank is bounded by $4(p-1)$. 
A similar reasoning can be applied to the matrix $A_n^R$, and we have
\begin{equation*}
A_n^R=n^\alpha (T^R_n+R^R_n),
\end{equation*}
where 
\begin{equation*}
T^R_n:=\left[ \D{nx}{}{n}{\alpha}\phi_p\left(\frac{p+1}{2}+i-j\right)\right]_{i,j=1}^{n+p-2},
\end{equation*}
and $R^R_n$ is a matrix whose rank is bounded by $4(p-1)$. 
As a consequence, the coefficient matrix $A_n^{p,\alpha}$ inherits the Toeplitz plus rank correction structure and can be written as follows:
\begin{equation}\label{eq:coeff_matrix}
A_n^{p,\alpha}=\frac{1}{2\cos({\pi\alpha}/{2})}(A_{n}^L+A_n^R)=n^\alpha (T_n^{p,\alpha}+R_n^{p,\alpha}),
\end{equation}
with 
\begin{equation}\label{eq:toeplitz-part}
T_n^{p,\alpha}:=\frac{1}{2\cos({\pi\alpha}/{2})}(T^L_n+T^R_n), \quad
R_n^{p,\alpha}:=\frac{1}{2\cos({\pi\alpha}/{2})}(R^L_n+R^R_n).
\end{equation}

In Section~\ref{sec:symbol} we will show that the symbol of $\{n^{-\alpha}A_n^{p,\alpha}\}_n$ coincides with the symbol of $\{T_n^{p,\alpha}\}_n$, denoted by $f^{p,\alpha}$, but first we discuss some properties of this function in the next section.

\section{Properties of the function $f^{p,\alpha}$}\label{sec:symbol-properties}
We start with a theorem that provides an explicit expression of the generating function $f^{p,\alpha}$ of the Toeplitz matrix $T_n^{p,\alpha}$, and whose proof uses the results obtained in Section~\ref{sec:frac-cardinal}.
\begin{theorem}\label{thm:symbol}
	Let $T_n^{p,\alpha}$ be defined as in \eqref{eq:toeplitz-part}. Then, $T_n^{p,\alpha}=T_{n+p-2}(f^{p,\alpha})$ with
	\begin{equation}
	\label{eq:frac-symbol}
	f^{p,\alpha}(\theta) = \sum_{l\in\ZZ} |\theta+2l\pi|^{\alpha} \left(\frac{\sin(\theta/2+l\pi)}{\theta/2+l\pi}\right)^{p+1}.
	\end{equation}
\end{theorem}
\begin{proof}
	From its construction it is clear that $T_n^{p,\alpha}$ is a Toeplitz matrix of dimension $n+p-2$.
	According to the definition in \eqref{eq:toeplitz}, the entries $f_k$ of this matrix are given by
	\begin{align*}
	f_k &= \frac{1}{2\cos(\pi \alpha/2)} \left(\Dl{\alpha}\phi_{p}\left(\frac{p+1}{2}-k\right)+\Dr{\alpha}\phi_{p}\left(\frac{p+1}{2}-k\right)\right).
	\end{align*}
	
	We differentiate the cases of odd and even degree $p$. We start by proving the expression \eqref{eq:frac-symbol} of the generating function $f^{p,\alpha}$ for $p = 2q+1$. Using Theorem~\ref{thm:inner-product-mix} (and its proof) with $\alpha=\alpha_1+\alpha_2$ and $q=p_1=p_2$, we have
	\begin{align*}
	2\cos(\pi \alpha/2) f_k &= \Dl{\alpha}\phi_{2q+1}\left(q+1-k\right)+\Dr{\alpha}\phi_{2q+1}\left(q+1-k\right) \\
	&= \frac{1}{2\pi}\int_\RR \left[(\i\theta)^{\alpha}+(-\i\theta)^{\alpha}\right]\left(\frac{1-\e^{-\i\theta}}{\i\theta}\right)^{q+1}\left(\frac{\e^{\i\theta}-1}{\i\theta}\right)^{q+1}\e^{-\i k\theta}\,\dd \theta \\
	&= \frac{1}{2\pi}\int_\RR 2|\theta|^{\alpha}\cos(\pi \alpha/2) \left|\frac{1-\e^{-\i\theta}}{\theta}\right|^{2q+2}\e^{-\i k\theta}\,\dd \theta.
	\end{align*}
	Set
	$$
	w(\theta):=|\theta|^{\alpha}\left|\frac{1-\e^{-\i\theta}}{\theta}\right|^{2q+2}
	=|\theta|^{\alpha}\left(\frac{\sin(\theta/2)}{\theta/2}\right)^{2q+2}.
	$$
	Then,
	\begin{align*}
	f_k &= \frac{1}{2\pi}\int_\RR w(\theta)\e^{-\i k\theta}\,\dd \theta
	=\sum_{l\in\ZZ}\frac{1}{2\pi}\int_{(2l-1)\pi}^{(2l+1)\pi} w(\theta)\e^{-\i k\theta}\,\dd \theta \\
	&=\sum_{l\in\ZZ}\frac{1}{2\pi}\int_{-\pi}^{\pi} w(\theta+2l\pi)\e^{-\i k\theta}\,\dd \theta
	=\frac{1}{2\pi}\int_{-\pi}^{\pi} \left[\sum_{l\in\ZZ}w(\theta+2l\pi)\right]\e^{-\i k\theta}\,\dd \theta.
	\end{align*}
	The expression \eqref{eq:frac-symbol} of the generating function $f^{p,\alpha}$ follows from \eqref{eq:toeplitz-coeff} for $p = 2q+1$.
	
	Now we consider even degree $p = 2q$. Using again Theorem~\ref{thm:inner-product-mix} (and its proof) with $\alpha=\alpha_1+\alpha_2$ and $q=p_1+1=p_2$, we have
	\begin{align*}
	2\cos(\pi \alpha/2) f_k &= \Dl{\alpha}\phi_{2q}\left(q+1-k-1/2\right)+\Dr{\alpha}\phi_{2q}\left(q+1-k-1/2\right) \\
	&= \frac{1}{2\pi}\int_\RR \left[(\i\theta)^{\alpha}+(-\i\theta)^{\alpha}\right]\left(\frac{1-\e^{-\i\theta}}{\i\theta}\right)^{q}\left(\frac{\e^{\i\theta}-1}{\i\theta}\right)^{q+1}\e^{-\i (k+1/2)\theta}\,\dd \theta \\
	&= \frac{1}{2\pi}\int_\RR 2|\theta|^{\alpha}\cos(\pi \alpha/2) \left|\frac{1-\e^{-\i\theta}}{\theta}\right|^{2q}\left(\frac{\e^{\i\theta/2}-\e^{-\i\theta/2}}{\i\theta}\right)\e^{-\i k\theta}\,\dd \theta.
	\end{align*}
	Here,
	\begin{equation*}
	w(\theta):=|\theta|^{\alpha}\left|\frac{1-\e^{-\i\theta}}{\theta}\right|^{2q}\left(\frac{\e^{\i\theta/2}-\e^{-\i\theta/2}}{\i\theta}\right)
	=|\theta|^{\alpha}\left(\frac{\sin(\theta/2)}{\theta/2}\right)^{2q+1}.
	\end{equation*}
	Following then a similar argument as in the odd degree case, we arrive at the expression \eqref{eq:frac-symbol} of the generating function $f^{p,\alpha}$ for $p = 2q$.
\end{proof}
\begin{remark}
	The proof of Theorem~\ref{thm:symbol} remains valid for $\alpha\in[0,1)\cup(1,2]$.
\end{remark}

Starting from \eqref{eq:frac-symbol} and applying the same line of arguments as in the proofs of \cite[Lemmas~3.4 and~3.6]{DGMSS}, we obtain the following results for $f^{p,\alpha}(\theta)$.
\begin{theorem}\label{thm:bounds-frac-symb}
	Let $f^{p,\alpha}$ be as in \eqref{eq:frac-symbol}. Then, $f^{p,\alpha}(\theta)=f^{p,\alpha}(-\theta)$, and for $p>\alpha $,
	\begin{equation}
	|\theta|^\alpha \left(\frac{\sin(\theta/2)}{\theta/2}\right)^{p+1}\leq f^{p,\alpha}(\theta)\leq|\theta|^\alpha\left(\frac{\sin(\theta/2)}{\theta/2}\right)^{p+1}+
	C_{p,\alpha}\left(\sin(\theta/2)\right)^{p+1},
	\label{eq:bounds-frac-symb}
	\end{equation}
	where $C_{p,\alpha}$ is a constant depending on $p$ and $\alpha$ and $\theta \in [0,\pi]$. Moreover,
	\begin{equation}
	\label{eq:val-pi-frac-symb}
	\frac{f^{p,\alpha}(\pi)}{\max_\theta f^{p,\alpha}(\theta)}\leq 
	\frac{f^{p,\alpha}(\pi)}{f^{p,\alpha}(\pi/2)}\leq 2^{\frac{2\alpha+1-p}{2}}.
	\end{equation}
\end{theorem}

From the bounds in \eqref{eq:bounds-frac-symb} we can immediately deduce the vanishing properties of $f^{p,\alpha}$.

\begin{corollary}\label{cor:zero_of_f}
	Let $f^{p,\alpha}$ be as in \eqref{eq:frac-symbol}. Then, $f^{p,\alpha}$ is non-negative for $\theta\in[-\pi,\pi]$, and it only vanishes at $\theta=0$ where it has a zero of order $\alpha$.
\end{corollary}

Let $\tilde{f}^{p,\alpha}:=f^{p,\alpha}/{\max_\theta f^{p,\alpha}(\theta)}$ be the normalized version of $f^{p,\alpha}$. The inequality in \eqref{eq:val-pi-frac-symb} shows that $\tilde{f}^{p,\alpha}(\theta)$ converges exponentially to zero at $\theta=\pm\pi$ for increasing $p$. Hence, we say that $f^{p,\alpha}$ has a numerical zero at $\pm\pi$ for large $p$.

\begin{remark}\label{rem:alpha}
	The upper bound in \eqref{eq:val-pi-frac-symb} depends not only on $p$ but also on $\alpha$. In this view, the decay at $\pm\pi$ of $\tilde{f}^{p,\alpha}$ is expected to become faster as $\alpha$ approaches 1.
\end{remark}

In the following propositions we bound $f^{p,\alpha}(\theta)$ in terms of $f^{p,0}(\theta)$ and $f^{p,2}(\theta)$ for high enough value of $|\theta|$.

\begin{proposition}\label{prop:bound_odd}
	For $p$ odd, we have
	\begin{equation} \label{eq:bound-high-freq}
	f^{p,0}(\theta)\leq f^{p,\alpha}(\theta)\leq f^{p,2}(\theta), 
	\quad |\theta|\in [1,\pi].
	\end{equation}
\end{proposition}
\begin{proof}
	Since
	\begin{equation*}
	1=|\theta+2l\pi|^0\leq|\theta+2l\pi|^{\alpha}\leq|\theta+2l\pi|^2, 
	\quad l\in\ZZ, \quad |\theta|\geq1,
	\end{equation*}
	and
	\begin{equation*}
	\left(\frac{\sin(\theta/2+l\pi)}{\theta/2+l\pi}\right)^{2q+2}\geq0, \quad l\in\ZZ,
	\end{equation*}
	it is clear from the definition of $f^{p,\alpha}$ in \eqref{eq:frac-symbol} that
	\eqref{eq:bound-high-freq} holds for $p=2q+1$.
\end{proof}

\begin{proposition}\label{prop:bound_even}
	For $p$ even and $p>\alpha$, we have
	\begin{equation} \label{eq:bound-high-freq-even}
	f^{p,0}(\theta)\leq f^{p,\alpha}(\theta), \quad |\theta|\in [a,\pi],
	\end{equation}
	where
	$$
	a:=\left(\frac{\pi^4}{48}\right)^{1/\alpha}.
	$$
\end{proposition}
\begin{proof}
	Let $p=2q>\alpha$. It is easy to check that 
	\begin{equation*}
	f^{p,\alpha}(\theta)
	= |\theta|^\alpha\left(\frac{\sin(\theta/2)}{\theta/2}\right)^{p+1} + \left(2\sin(\theta/2)\right)^{p+1}r^{p,\alpha}(\theta),
	\end{equation*}
	where
	\begin{equation*}
	r^{p,\alpha}(\theta) :=
	\sum_{k=1}^{\infty}(-1)^k\left[\frac{1}{(2k\pi+\theta)^{p+1-\alpha}}
	-\frac{1}{(2k\pi-\theta)^{p+1-\alpha}}\right].
	\end{equation*}
	With the same line of arguments as in the proof of \cite[Lemma~A.2]{DGMSS} we deduce that $r^{p,\alpha}(\theta)$
	is a strictly increasing function, which implies that $r^{p,\alpha}(\pi)\geq r^{p,\alpha}(\theta)>r^{p,\alpha}(0)=0$ for $\theta\in(0,\pi]$. Moreover, from the same lemma we know
	\begin{equation*}
	r^{p,0}(\theta) \leq \left(\frac{\pi^4}{48}-1\right)\frac{1}{\pi^{p+1}},
	\quad\theta\in[0,\pi].
	\end{equation*}
	From the above bounds we get
	\begin{align*}
	f^{p,\alpha}(\theta)-f^{p,0}(\theta) &=
	(|\theta|^\alpha-1)\left(\frac{\sin(\theta/2)}{\theta/2}\right)^{p+1} + \left(2\sin(\theta/2)\right)^{p+1}(r^{p,\alpha}(\theta)-r^{p,0}(\theta)) \\
	&\geq \left(2\sin(\theta/2)\right)^{p+1}\left[\frac{|\theta|^\alpha-1}{\theta^{p+1}}-\left(\frac{\pi^4}{48}-1\right)\frac{1}{\pi^{p+1}}\right] \\
	&\geq \left(\frac{2\sin(\theta/2)}{\pi}\right)^{p+1}\left[|\theta|^\alpha-1-\left(\frac{\pi^4}{48}-1\right)\right],
	\end{align*}
	for $\theta\in[1,\pi]$. Hence,
	\begin{equation*}
	f^{p,\alpha}(\theta)-f^{p,0}(\theta)\geq0, \quad 
	\theta\geq\left(\frac{\pi^4}{48}\right)^{1/\alpha},
	\end{equation*}
	which concludes the proof.
\end{proof}

In our final proposition we explicitly state that $f^{p,\alpha}$ is the symbol of the matrix-sequence $\{T_n^{p,\alpha}\}_n$.
\begin{proposition} \label{prop:distr-T}
	The Toeplitz matrix $T_n^{p,\alpha}$ defined in \eqref{eq:toeplitz-part} is symmetric and
	\begin{equation}\label{eq:distr-T}
	\{T_n^{p,\alpha}\}_n\sim_\lambda(f^{p,\alpha},[-\pi,\pi]),
	\end{equation}
	where $f^{p,\alpha}$ is given in \eqref{eq:frac-symbol}.
\end{proposition}
\begin{proof}
	From Theorem~\ref{thm:bounds-frac-symb} we know that $f^{p,\alpha}$ is an even real-valued function, so the matrix $T_n^{p,\alpha}=T_{n+p-2}(f^{p,\alpha})$ is symmetric. The spectral distribution of $\{T_n^{p,\alpha}\}_n=\{T_{n+p-2}(f^{p,\alpha})\}_n$ follows from Theorem~\ref{szego}.
\end{proof}

We end this section by summarizing all the discussed properties of the symbol $f^{p,\alpha}$ and highlighting what is their role in the design of an ad hoc solver for a linear system associated with $T_n^{p,\alpha}$ (see Remark~\ref{rem:mim}).
We have shown that $f^{p,\alpha}$ is equipped with the following three properties:
\begin{enumerate}
	\item[{(a)}] it has a single zero at $0$ of order $\alpha$ (Corollary~\ref{cor:zero_of_f}); 
	\item[{(b)}] it presents an exponential decay to zero at $\pi$ for increasing $p$ that becomes faster as $\alpha$ approaches $1$ (Theorem~\ref{thm:bounds-frac-symb} and Remark~\ref{rem:alpha});
	\item[{(c)}] it is bounded in the proximity of $\pi$ by $f^{p,0}$ (Propositions~\ref{prop:bound_odd} and~\ref{prop:bound_even}).
\end{enumerate}
Properties (a)--(b) give us a clear picture of what are the conditioning peculiarities of the matrix $T_n^{p,\alpha}$. Specifically, they say that $T_n^{p,\alpha}$ is poorly conditioned both in the low frequencies (with a conditioning that grows as $n^\alpha$) and in the high frequencies (with a deterioration that is driven both by $p$ and $\alpha$). Moreover, property (c) ``isolates'' the source of ill-conditioning in the high frequencies induced by $p$, meaning the symbol behaves like $f^{p,0}$ in the proximity of $\pi$, with $f^{p,0}$ a positive function well-separated from zero.

\begin{remark}\label{rem:mim}
	Based on what has been done in \cite{cmame2,sinum,mazza2016,sisc}, all this knowledge can be used for the design of an ad hoc solver for a linear system associated with $T_n^{p,\alpha}$. For instance, from (a) we can infer that a multigrid method with a standard choice of both prolongator and restrictor is able to cope with the standard ill-conditioning in the low frequency subspace, while from (c) we get hints on how to define a smoother that works in the subspace of high frequencies where there exists the ill-conditioning induced by $p$. 
\end{remark}

\section{Spectral symbol of $\{n^{-\alpha}A_n^{p,\alpha}\}_n$}\label{sec:symbol}
This section is devoted to the computation of the symbol of the matrix-sequence $\{n^{-\alpha}A_n^{p,\alpha}\}_n$. 
As we have already anticipated, it turns out that the symbol of $\{n^{-\alpha}A_n^{p,\alpha}\}_n$ coincides with the symbol of the Toeplitz part $\{T_n^{p,\alpha}\}_n$. 
The spectral distribution of $\{n^{-\alpha}A_n^{p,\alpha}\}_n$ is given in Theorem~\ref{thm:spectral-A}. Its proof uses Corollary~\ref{cor:quasi-herm} and needs several preliminary results. 

\begin{sloppypar}
For a given matrix $X:=[x_{ij}]_{i,j=1}^m\in\CC^{m\times m}$, we denote by $\|X\|_1:=\max_{j=1,\ldots,m}\sum_{i=1}^m |x_{ij}|$ and $\|X\|_\infty:=\max_{i=1,\ldots,m}\sum_{j=1}^m |x_{ij}|$, the induced  1- and infinity-norm, respectively.
\end{sloppypar}

\begin{lemma}\label{lem:elem-bound-AL}
	Let $A_{n}^L$ be defined as in \eqref{eq:ALR}. For $i,j=2,\dots, n+p-1$ we have
	\begin{equation}\label{eq:elem-bound-AL}
	|(A_n^{L})_{i-1,j-1}|\leq \begin{cases}
	0, & \eta_{i}\leq \xi_j,
	\\[0.2cm]
	c_{p,\alpha}^{L}\, n^\alpha, & \xi_j<\eta_i\leq \xi_{j+p+1}+\frac{1}{n},
	\\[0.2cm]
	c_{p,\alpha}^{L}\, (\eta_i-\xi_{j+p+1})^{-\alpha}, & \xi_{j+p+1}+\frac{1}{n} < \eta_i,
	\end{cases}
	\end{equation}
	where $c_{p,\alpha}^{L}$ is a constant depending on $p$ and $\alpha$.
\end{lemma}
\begin{proof}
	From the properties of fractional derivatives \eqref{eq:caputo}--\eqref{eq:relRLC} and the B-spline properties \eqref{eq:spline-support}--\eqref{eq:spline-frac}
	it follows that for $j=2$,
	\begin{equation}
	\label{first-matrix-element}
	(A_n^{L})_{i-1,1}= 
	\frac{1}{\Gamma(2-\alpha)}\int_{0}^{\min{(\eta_i, \xi_{p+3})}}(\eta_i-y)^{1-\alpha}(N^p_2)''(y)\,\dd y + \frac{pn}{\Gamma(2-\alpha)}(\eta_i)^{1-\alpha},
	\end{equation}
	and for $j=3,\ldots,n+p-1$,
	\begin{equation}
	\label{matrix-elements}
	(A_n^{L})_{i-1,j-1}=\begin{cases}
	0, & \eta_{i}\leq \xi_j,
	\\
	\displaystyle\frac{1}{\Gamma(2-\alpha)}\int_{\xi_j}^{\min{(\eta_i, \xi_{j+p+1})}}(\eta_i-y)^{1-\alpha}(N^p_j)''(y)\,\dd y, & \text{otherwise}.
	\end{cases}
	\end{equation}
	We remark that $\eta_i\in(0,1)$, $\eta_i<\eta_{i+1}$, and   
	$\xi_{j+p+1}-\xi_j\leq \frac{p+1}{n}$.
	In the following, we address the three different cases in \eqref{eq:elem-bound-AL} separately.
	
	If $\eta_i\leq \xi_j$, then it is clear that $(A_n^{L})_{i-1,j-1}=0$ for $j=3,\dots, n+p-1$. 
	Note that $j=2$ is not involved in this case for any $i$ because $\xi_2=0<\eta_i$.
	
	If $\xi_j<\eta_i\leq \xi_{j+p+1}+\frac{1}{n}$, then
	\begin{equation*}
	(\eta_i-y)\leq \frac {p+2}{n}, \quad y\in [\xi_j, \min{(\eta_i, \xi_{j+p+1})}].
	\end{equation*}
	Using \eqref{eq:spline-diff-bound}, from \eqref{matrix-elements} we get for $j=3,\dots, n+p-1$,
	\begin{align*}
	|(A_n^{L})_{i-1,j-1}|&\leq \frac{4p(p-1)n^2}{\Gamma(2-\alpha)}\int_{\xi_j}^{\min{(\eta_i, \xi_{j+p+1})}}(\eta_i-y)^{1-\alpha}\,\dd y\\
	&\leq \frac{4p(p-1)n^2}{\Gamma(3-\alpha)}\left(\frac{p+2}{n}\right)^{2-\alpha}.
	\end{align*}
	When $j=2$ we have $\xi_2=0$ and $\frac{1}{pn}=\eta_2\leq\eta_i\leq \xi_{p+3}+\frac{1}{n}\leq\frac{3}{n}$. Then, we find in a similar way that for $j=2$,
	\begin{equation*}
	|(A_n^{L})_{i-1,1}|\leq \frac{4p(p-1)n^2}{\Gamma(3-\alpha)}\left(\frac{3}{n}\right)^{2-\alpha}+ \frac{pn}{\Gamma(2-\alpha)}\left(\frac{1}{pn}\right)^{1-\alpha}.
	\end{equation*}
	
	We now look at the case  $\xi_{j+p+1}+\frac{1}{n}< \eta_i$. 
	This case can only happen for $2\leq j< n$ because when $j\geq n$ we have $\xi_{j+p+1}=1>\eta_i$.
	Given $y\in [\xi_j, \xi_{j+p+1}]$, we consider the Taylor expansion of $(\eta_i-y)^{1-\alpha}$ at $\xi_j$, producing
	\begin{equation}
	\label{Taylor-1}
	(\eta_i-y)^{1-\alpha} = (\eta_i-\xi_j)^{1-\alpha}-(1-\alpha)(\eta_i-\omega_{i,j}(y))^{-\alpha}(y-\xi_j), 
	\end{equation}
	for some $\omega_{i,j}(y)\in (\xi_j, \xi_{j+p+1})$.
	Substituting \eqref{Taylor-1} in \eqref{matrix-elements} results in
	\begin{align*}
	|(A_n^{L})_{i-1,j-1}| &\leq \frac{1}{\Gamma(2-\alpha)}\left |(\eta_i-\xi_j)^{1-\alpha}\int_{\xi_j}^{\xi_{j+p+1}}(N^p_j)''(y)\,\dd y\right|
	\\
	&\quad+\frac{\alpha-1}{\Gamma(2-\alpha)}\int_{\xi_j}^{\xi_{j+p+1}}(\eta_i-\omega_{i,j}(y))^{-\alpha}(y-\xi_j)|(N^p_j)''(y)|\,\dd y.
	\end{align*}
	Observe that $(N^p_j)'(\xi_{j})=(N^p_j)'(\xi_{j+p+1})=0$ for $3\leq j\leq n+p-2$, and $(\eta_i-\omega_{i,j}(y))>(\eta_i-\xi_{j+p+1})$. Then, recalling the bound in \eqref{eq:spline-diff-bound}, we obtain for $j=3,\dots,n-1$,
	\begin{align*}
	|(A_n^{L})_{i-1,j-1}| &\leq \frac{1}{\Gamma(2-\alpha)}\left |(\eta_i-\xi_j)^{1-\alpha}((N^p_j)'(\xi_{j+p+1})-(N^p_j)'(\xi_j))\right|
	\\
	&\quad +\frac{\alpha-1}{\Gamma(2-\alpha)}4p(p-1)n^2\int_{\xi_j}^{\xi_{j+p+1}}(y-\xi_j)\,\dd y\,(\eta_i-\xi_{j+p+1})^{-\alpha}
	\\
	&\leq \frac{\alpha-1}{\Gamma(2-\alpha)}2p(p-1)n^2\left(\frac{p+1}{n}\right)^2(\eta_i-\xi_{j+p+1})^{-\alpha}.
	\end{align*}
	Substituting \eqref{Taylor-1} in \eqref{first-matrix-element} and observing that $(N^p_2)'(0)=np$, $(N^p_2)'(\xi_{p+3})=0$, we find with a similar argument that for $j=2$,
	\begin{align*}
	|(A_n^{L})_{i-1,1}| &\leq \frac{1}{\Gamma(2-\alpha)}\left |(\eta_i)^{1-\alpha}((N^p_2)'(\xi_{p+3})-(N^p_2)'(0))+np(\eta_i)^{1-\alpha}\right|
	\\
	&\quad+\frac{\alpha-1}{\Gamma(2-\alpha)}4p(p-1)n^2\int_{0}^{\xi_{p+3}}y\,\dd y\, (\eta_i-\xi_{j+3})^{-\alpha}
	\\
	&\leq \frac{\alpha-1}{\Gamma(2-\alpha)}2p(p-1)n^2\left(\frac{2}{n}\right)^2 (\eta_i-\xi_{j+3})^{-\alpha}.
	\end{align*} 
	This concludes the proof.
\end{proof}

\begin{lemma}\label{lem:AL_norm}
	Let $A_{n}^L$ be defined as in \eqref{eq:ALR}. We have
	\begin{equation*}
	\|n^{-\alpha}A_{n}^L\|_{q}\leq C_{p,\alpha}^{L}, \quad q\in\{1,2,\infty\},
	\end{equation*}
	where $C_{p,\alpha}^{L}$ is a constant depending on $p$ and $\alpha$.
\end{lemma}
\begin{proof}
	We first consider the infinity-norm
	\begin{align*}
	\|n^{-\alpha}A_{n}^L\|_\infty &=n^{-\alpha}\max_{i=2\ldots,n+p-1}\sum_{j=2}^{n+p-1}|(A_n^{L})_{i-1,j-1}|.
	\end{align*}
	The entries $|(A_n^{L})_{i-1,j-1}|$, $i,j=2, \dots n+p-1$, can be bounded thanks to the results of Lemma~\ref{lem:elem-bound-AL}. We observe that for any fixed $i$,
	\begin{itemize}  
		\item the number of indices in $\{j: \xi_j<\eta_i\leq \xi_{j+p+1}+\frac{1}{n} \}$ is bounded by $p+2$;
		\item for $j=n,\dots, n+p-1$ we have $\xi_{j+p+1}=1$, thus  either  $\eta_i\leq \xi_j$ or $ \xi_j<\eta_i\leq \xi_{j+p+1}+\frac{1}{n}$;
		\item if $\eta_i> \xi_{j+p+1}+\frac{1}{n}$, then $2\leq j\leq n-1$ and
		\begin{equation*}
		\eta_i-\xi_{j+p+1}=\eta_i-\frac{j}{n}\geq \frac{\ell_{i,j}}{n},
		\end{equation*}
		where $\ell_{i,j}:=\lfloor n\eta_i-j\rfloor$. Note that 
		$\ell_{i,j}\geq 1$ and $\ell_{i,j}=\ell_{i,j+1}+1$, so
		$$
		\sum_{j:\,\eta_i> \xi_{j+p+1}+\frac{1}{n}} (\ell_{i,j})^{-\alpha}\leq \sum_{\ell=1}^\infty\ell^{-\alpha}=\zeta(\alpha),
		$$
		with $\zeta(\alpha)$ the Riemann zeta function evaluated at $\alpha$. The series $\sum_{\ell=1}^\infty\ell^{-\alpha}$ is convergent for $\alpha\in (1,2)$.
	\end{itemize}
	As a consequence, taking into account Lemma~\ref{lem:elem-bound-AL}, for any fixed $i$ we have
	\begin{align*}
	n^{-\alpha}\sum_{j=2}^{n+p-1}|(A_n^{L})_{i-1,j-1}| &\\
	&\hspace{-40pt}\leq
	n^{-\alpha}\left[
	\sum_{j:\, \xi_j<\eta_i\leq\xi_{j+p+1}+\frac{1}{n}}|(A_n^{L})_{i-1,j-1}|
	+ \sum_{j:\, \eta_i>\xi_{j+p+1}+\frac{1}{n}}|(A_n^{L})_{i-1,j-1}|
	\right]
	\\
	&\hspace{-40pt}\leq  n^{-\alpha} \,c_{p,\alpha}^{L}\left[ (p+2) n^\alpha+\sum_{j:\, \eta_i>\xi_{j+p+1}+\frac{1}{n}}\left(\frac{\ell_{i,j}}{n}\right)^{-\alpha}\right]\\
	&\hspace{-40pt}\leq c_{p,\alpha}^{L}\left[ p+2 +\zeta(\alpha)\right].
	\end{align*}
	
	The bound for the 1-norm
	$$
	\|n^{-\alpha}A_{n}^L\|_1 = n^{-\alpha}\max_{j=2\ldots,n+p-1}\sum_{i=2}^{n+p-1}|(A_n^{L})_{i-1,j-1}|
	$$
	can be shown with a similar line of arguments, by observing that for any fixed $j$,
	\begin{itemize}  
		\item the number of indices in $\{i: \xi_j<\eta_i\leq \xi_{j+p+1}+\frac{1}{n} \}$ is bounded by $2p$;
		\item if $\eta_i> \xi_{j+p+1}+\frac{1}{n}$, then 
		\begin{equation*}
		\eta_i-\xi_{j+p+1}\geq \frac{\ell_{i,j}}{n},
		\end{equation*}
		where $\ell_{i,j}:=\lfloor n(\eta_i-\xi_{j+p+1})\rfloor$. Note that 
		$\ell_{i,j}\geq 1$ for all $i$ and $\ell_{i+1,j}=\ell_{i,j}+1$ for $p+1\leq i \leq n-1$, so
		$$
		\sum_{i:\,\eta_i> \xi_{j+p+1}+\frac{1}{n}} (\ell_{i,j})^{-\alpha}\leq 2(p-1)+\sum_{\ell=1}^\infty\ell^{-\alpha}=2(p-1)+\zeta(\alpha).
		$$
	\end{itemize}
	
	Finally, the bound for the spectral norm follows from the inequality
	\begin{equation*}
	\|n^{-\alpha}A_n^L\|_2\leq \sqrt{\|n^{-\alpha}A_n^L\|_\infty\|n^{-\alpha}A_n^L\|_1}
	\end{equation*}
	and the above results for the infinity-norm and 1-norm.
\end{proof}

A similar reasoning to the one adopted in the previous lemmas brings us to the following result. 

\begin{lemma}\label{lem:AR_norm}
	Let $A_{n}^R$ be defined as in \eqref{eq:ALR}. We have
	\begin{equation*}
	\|n^{-\alpha}A_{n}^R\|_q \leq C_{p,\alpha}^{R}, \quad q\in\{1,2,\infty\},
	\end{equation*}
	where $C_{p,\alpha}^{R}$ is a constant depending on $p$ and $\alpha$.
\end{lemma}

\begin{lemma}\label{lem:normR}
	Let $R_n^{p,\alpha}$ be defined as in \eqref{eq:toeplitz-part}. We have
	\begin{equation*}
	\|R_n^{p,\alpha}\|_2,\,\|R_n^{p,\alpha}\|_{1,\ast}\leq \widetilde{C}_{p,\alpha},
	\end{equation*}
	where $\widetilde{C}_{p,\alpha}$ is a constant depending on $p$ and $\alpha$.
\end{lemma}
\begin{proof}
	The relation in \eqref{eq:coeff_matrix} implies
	$$
	\|R_n^{p,\alpha}\|_2=\|n^{-\alpha}A_n^{p,\alpha}-T_n^{p,\alpha}\|_2\leq\|n^{-\alpha}A_n^{p,\alpha}\|_2+\|T_n^{p,\alpha}\|_2,$$ 
	and we recall from Section~\ref{sec:symbol-properties} that 
	$$
	\|T_n^{p,\alpha}\|_2=\|T_{n+p-2}(f^{p,\alpha})\|_2\leq\|f^{p,\alpha}\|_\infty<+\infty.
	$$ 
	Then, by Lemmas \ref{lem:AL_norm} and \ref{lem:AR_norm}, we arrive at 
	$$
	\|R_n^{p,\alpha}\|_2\leq \frac{1}{2\cos({\pi\alpha}/{2})}(C^L_{p,\alpha}+C^R_{p,\alpha})+\|f^{p,\alpha}\|_\infty.$$
	In addition, $\|R_n^{p,\alpha}\|_{1,\ast}\leq\text{rank}(R_n^{p,\alpha})\|R_n^{p,\alpha}\|_2$ and $\text{rank}(R_n^{p,\alpha})\le4(p-1)$. This completes the proof.
\end{proof}

We are now in a position to discuss the spectral distribution of $\{n^{-\alpha}A_n^{p,\alpha}\}_n$. 
\begin{theorem}\label{thm:spectral-A}
	Given $\{n^{-\alpha}A_n^{p,\alpha}\}_n$ with $A_n^{p,\alpha}$ as in \eqref{eq:coeff_matrix}, we have
	\begin{equation}\label{eq:distr-A}
	\{n^{-\alpha}A_n^{p,\alpha}\}_n\sim_\lambda(f^{p,\alpha},[-\pi,\pi]),
	\end{equation}
	where $f^{p,\alpha}$ is given in \eqref{eq:frac-symbol}.
\end{theorem}
\begin{proof}
	We prove this result by applying Corollary~\ref{cor:quasi-herm} with $X_n=T_n^{p,\alpha}$ and $Y_n=R_n^{p,\alpha}$. We first note that, because of Proposition~\ref{prop:distr-T}, condition (a) of Theorem~\ref{thm:quasi-herm} is satisfied. The other conditions in Corollary~\ref{cor:quasi-herm} hold by Lemma~\ref{lem:normR}, which proves the result \eqref{eq:distr-A}.
\end{proof}

\begin{remark} \label{rem:mim-A}
	Thanks to Theorem~\ref{thm:spectral-A}, the matrices $n^{-\alpha}A_n^{p,\alpha}$ and $T_n^{p,\alpha}$ are asymptotically spectrally equivalent, possibly up to few outliers. As a consequence, the arguments given in Remark~\ref{rem:mim} apply unchanged when the aim is solving a linear system whose coefficient matrix is  $n^{-\alpha}A_n^{p,\alpha}$ instead of $T_n^{p,\alpha}$.
\end{remark}

\begin{remark}\label{rem:adv-rea}
	\cmag{Let $A_n^{p,\alpha,\gamma,\rho}$ be the coefficient matrix corresponding to the B-spline collocation discretization of the advection-diffusion-reaction problem
		\begin{equation*}
		\begin{cases}
		\frac{\dd^\alpha u(x)}{\dd|x|^\alpha}+\gamma u'(x)+\rho u(x)=s(x), & x\in\Omega,\\
		u(x)=0, & x\in\RR\backslash\Omega,
		\end{cases}
		\end{equation*}
		with $\rho>0$ and $\gamma\in\RR$.
		From Theorem~\ref{thm:spectral-A}, in combination with Theorem~\ref{thm:quasi-herm} and \cite[Lemma~4.1]{DGMSS},
		we immediately deduce that}
	\begin{equation*}
	\{n^{-\alpha}A_n^{p,\alpha,\gamma,\rho}\}_n\sim_\lambda(f^{p,\alpha},[-\pi,\pi]).
	\end{equation*}
\end{remark}

\section{Numerical experiments} \label{sec:numerics}
In the following, we verify the spectral results obtained in Sections \ref{sec:symbol-properties} and \ref{sec:symbol} through several numerical experiments. We also provide a numerical study of the approximation behavior of the proposed polynomial B-spline collocation method for an arbitrary degree $p$. 

Let us start by illustrating that
\begin{itemize}
	\item the symbol $f^{p,\alpha}$ has a single zero at $0$ of order $\alpha$ and it presents an exponential decay to zero at $\pi$ for increasing $p$;
	\item the symbol $f^{p,\alpha}$ satisfies the bounds in \eqref{eq:bound-high-freq} for odd $p$, and the bound in \eqref{eq:bound-high-freq-even} for even $p$;
	\item relations \eqref{eq:distr-T} and \eqref{eq:distr-A} hold.
\end{itemize}
\cmag{Note that it suffices to consider the interval $[0,\pi]$ due to the symmetry of $f^{p,\alpha}$; see Theorem~\ref{thm:bounds-frac-symb}.}

Figure~\ref{fig:symbol} shows that, independently of $p$, the symbol $f^{p,\alpha}$ has a single zero at $0$ and the order of such zero increases up to 2 as $\alpha$ tends to 2. On the other hand, $f^{p,\alpha}$ presents a decay at $\pi$ as $p$ increases. We observe that such decay becomes faster when $\alpha$ decreases to 1, in accordance with Remark \ref{rem:alpha}. 
\begin{figure}[t]
	\centering
	\begin{subfigure}[$p=3$]
		{		\includegraphics[width=0.34\linewidth]{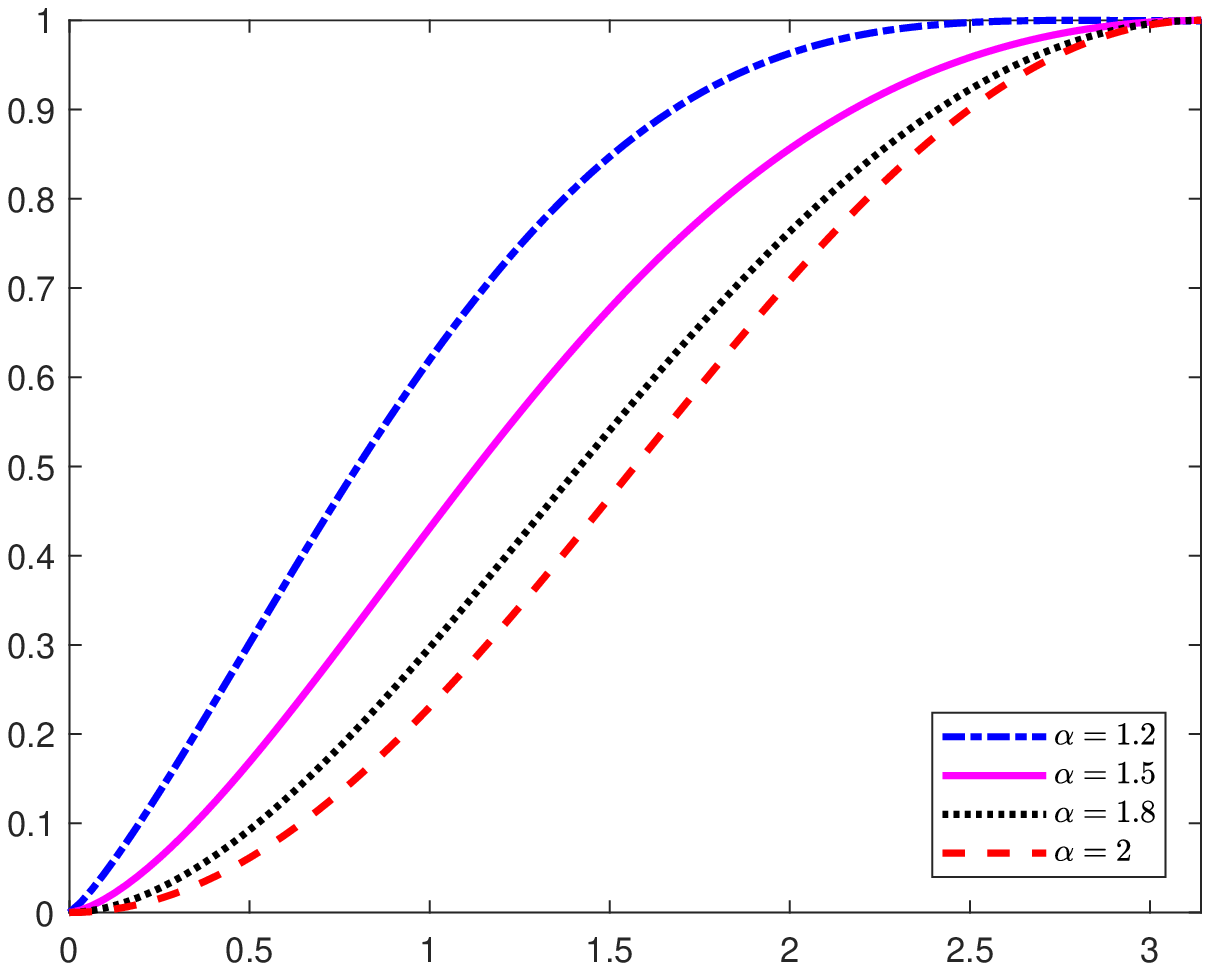}}
	\end{subfigure}\hspace*{-0.7cm}
	\begin{subfigure}[$p=5$]
		{	\includegraphics[width=0.34\linewidth]{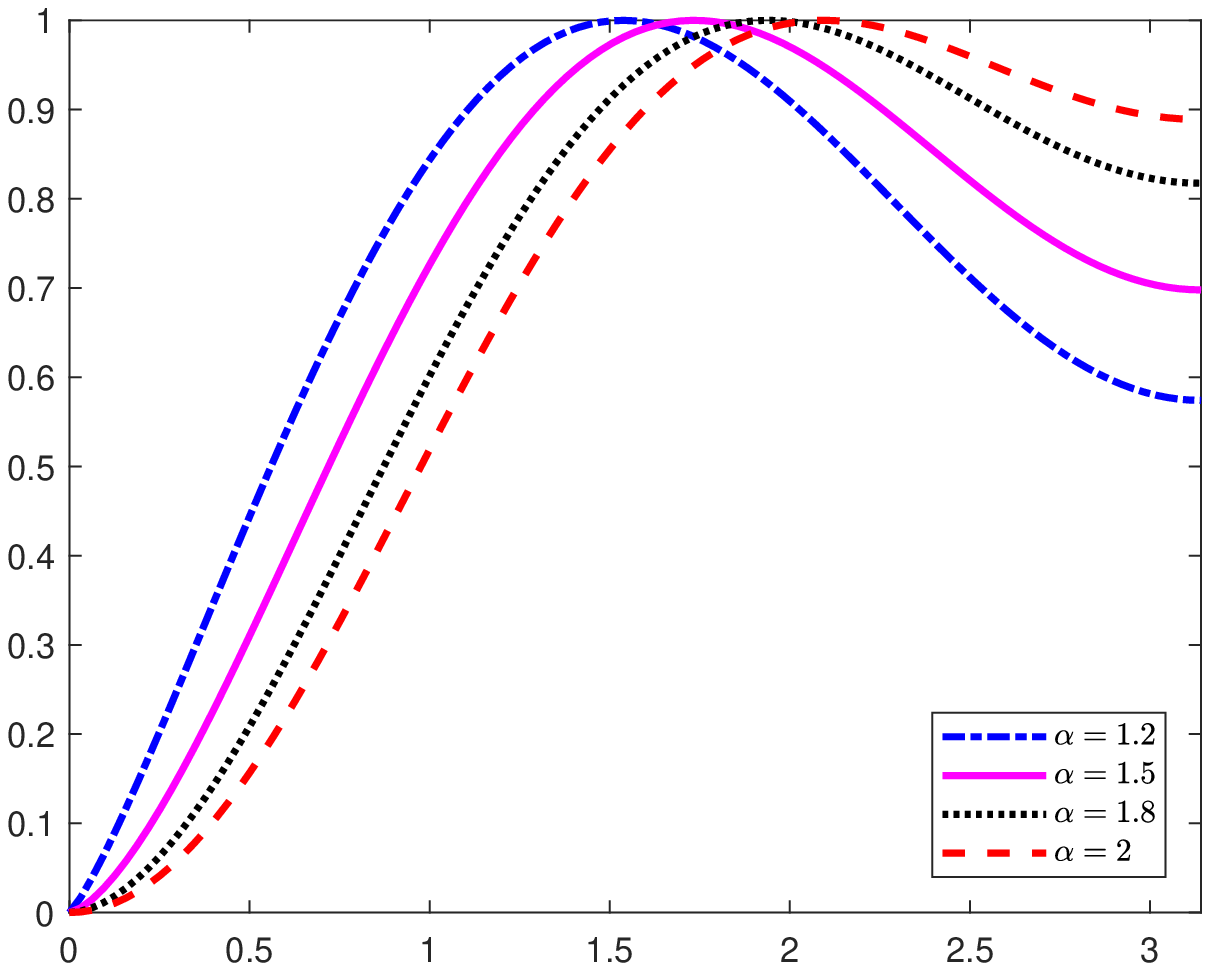}}
	\end{subfigure}\hspace*{-0.7cm}
	\begin{subfigure}[$p=8$]
		{		\includegraphics[width=0.34\linewidth]{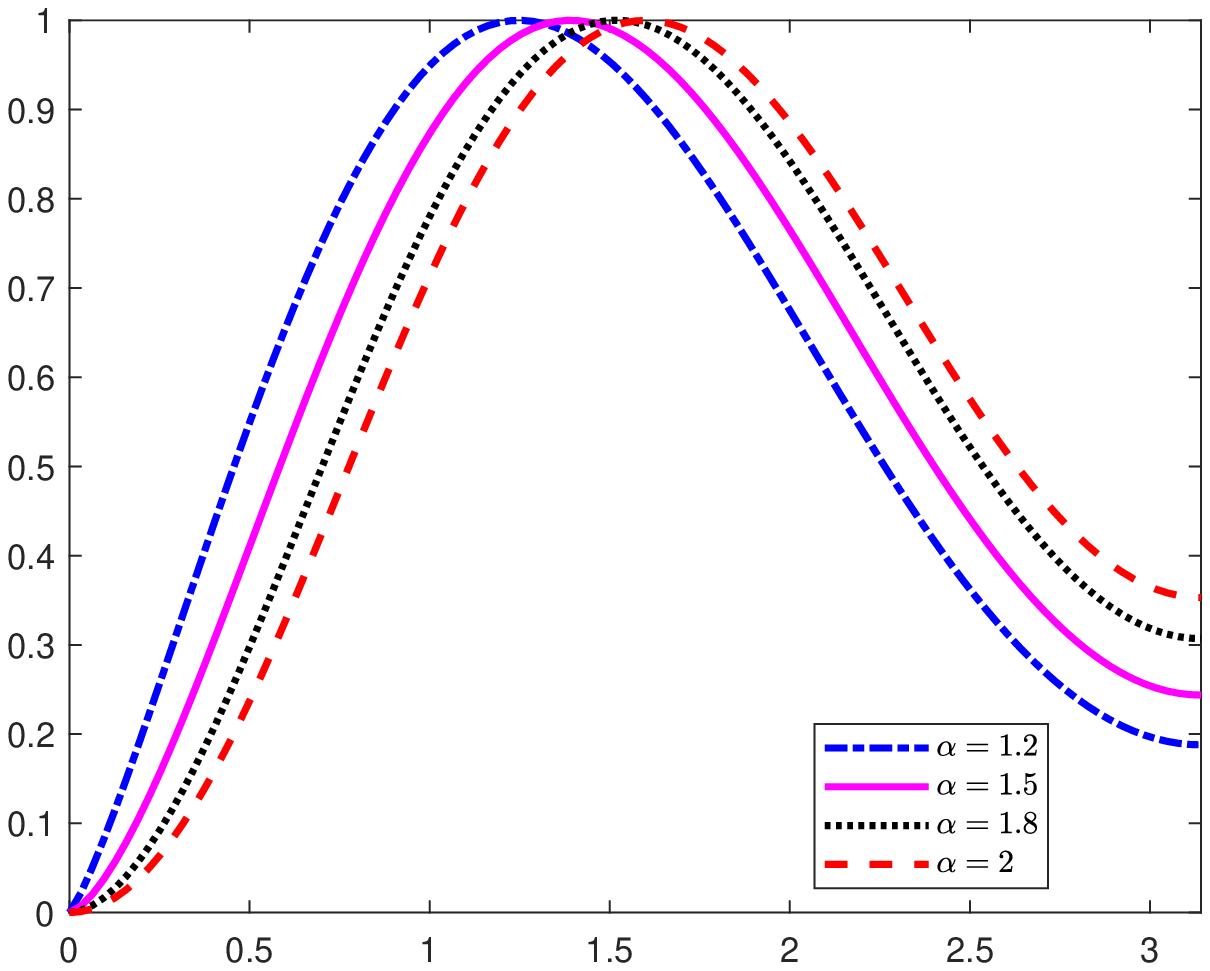}}
	\end{subfigure}
	\caption{Plot of $f^{p,\alpha}$ for $p=3,5,8$ and $\alpha=1.2,1.5,1.8,2$.} 
	\label{fig:symbol}
\end{figure}

In Figure~\ref{fig:bounds} we show that, fixing $\alpha=1.3$, the bounds in \eqref{eq:bound-high-freq} hold for $p=3$, and the one in \eqref{eq:bound-high-freq-even} holds for $p=4$. Observe that, despite relation \eqref{eq:bound-high-freq-even} is theoretically proven to be true for all $\theta\in[a,\pi]$, $a=\left(\frac{\pi^4}{48}\right)^{1/\alpha}$, it actually also holds for all $\theta\in[1,a]$, i.e., for all values on the left of the black vertical line $\theta=a$ shown in Figure \ref{fig:boundsb}.

\begin{figure}[t]
	\centering
	\begin{subfigure}[$p=3$]
		{		\includegraphics[width=0.46\linewidth]{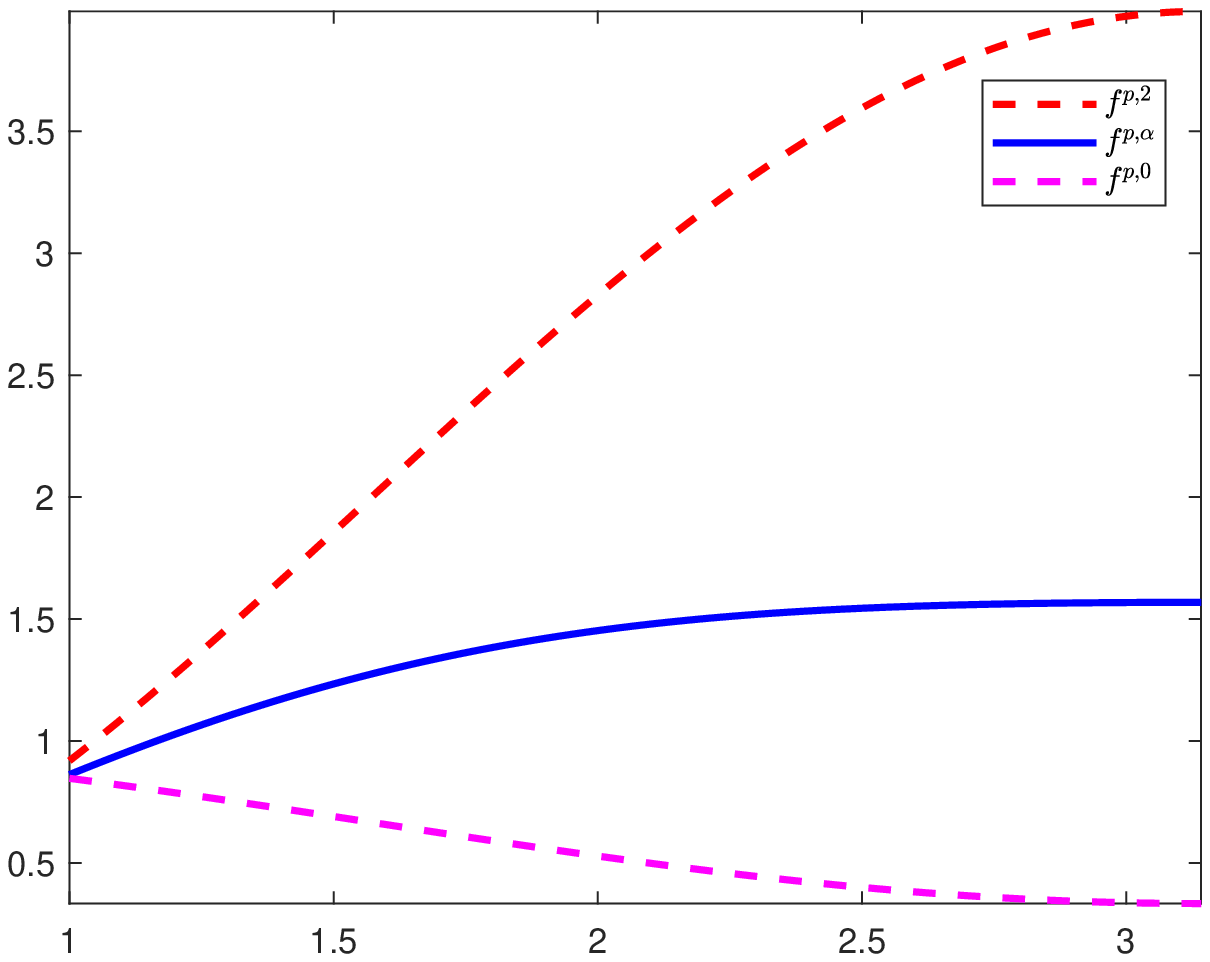}}
	\end{subfigure}
	\begin{subfigure}[$p=4$\label{fig:boundsb}]
		{		\includegraphics[width=0.46\linewidth]{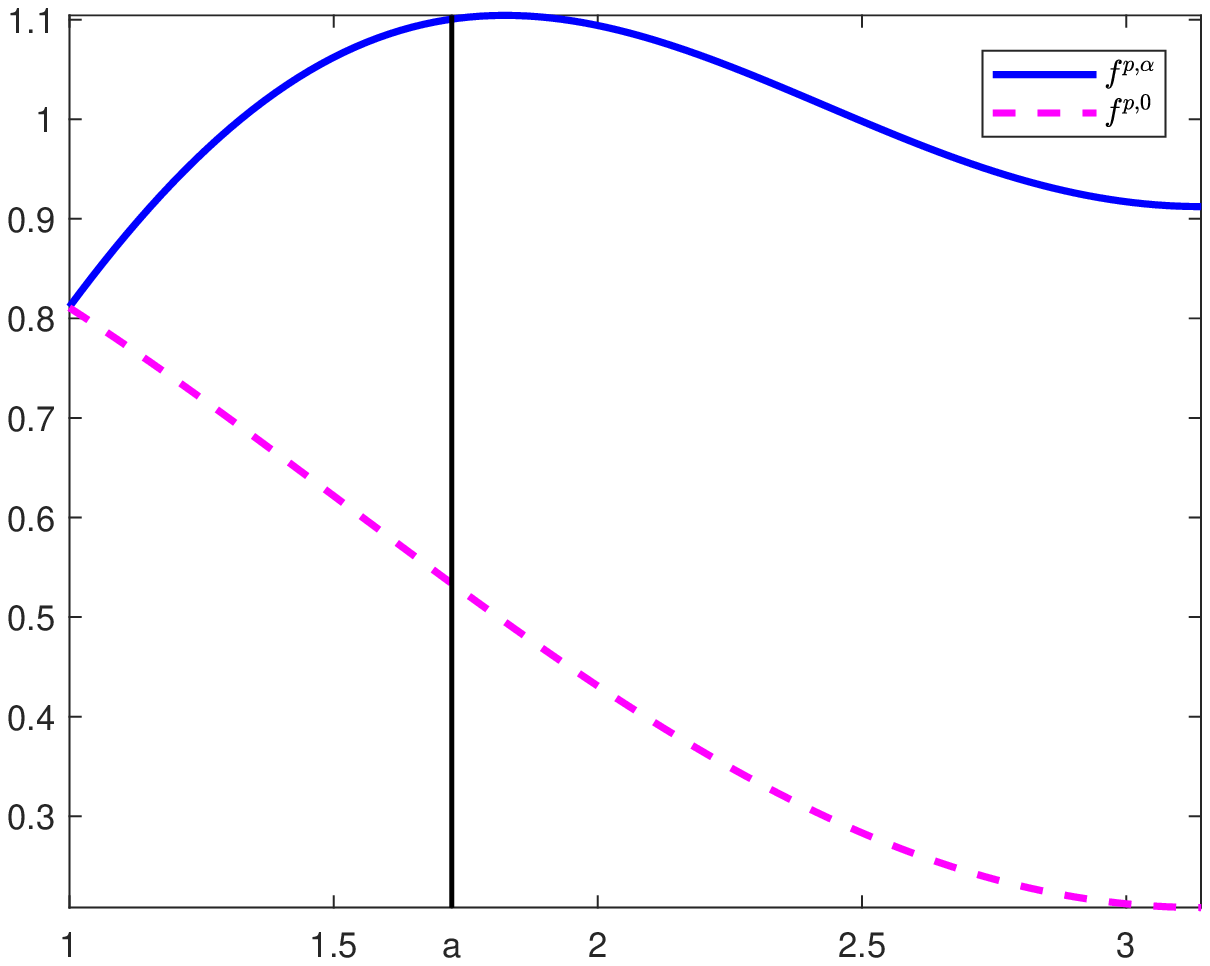}}
	\end{subfigure}
	\caption{(a) Check of the bound in \eqref{eq:bound-high-freq} which is valid for odd $p$, and (b) check of the bound in \eqref{eq:bound-high-freq-even} which is valid for even $p$. In both cases $\alpha$ has been fixed to $1.3$.}\label{fig:bounds}
\end{figure}

In order to numerically verify that relations \eqref{eq:distr-T} and \eqref{eq:distr-A} hold, for fixed $n$, $p$, we define the following equispaced grid on $[0,\pi]$:
\begin{equation*}
\varGamma:=\left\{\theta_{k}:=\frac{k\pi}{n+p-2}: k={1}, \dots,n+p-2\right\}.
\end{equation*}
Then, we compare the sampling of $f^{p,\alpha}$ on $\varGamma$ with the eigenvalues of both $T_n^{p,\alpha}$ and $n^{-\alpha}A_{n}^{p,\alpha}$. Both eigenvalues and sampling values have been ordered in ascending way. In all the numerical experiments, the entries of the coefficient matrix $A_n^{p,\alpha}$ have been computed using the Gauss-Jacobi-type quadrature rules introduced in \cite{Pang}. In Figure~\ref{fig:compeig1dot8_100} we fix $p=3$, $n=63$ and vary $\alpha\in\{1.2,1.8\}$. For both $T_n^{p,\alpha}$ and $n^{-\alpha}A_{n}^{p,\alpha}$ we experience a very good matching, which is in accordance with Proposition~\ref{prop:distr-T} and Theorem~\ref{thm:spectral-A}. 
However, we observe that in the case of $n^{-\alpha}A_{n}^{p,\alpha}$ there are few large eigenvalues that do not behave like the symbol; these are the \cmag{outliers and their number is independent of $n$.} As a further confirmation of Proposition~\ref{prop:distr-T} and Theorem~\ref{thm:spectral-A}, we obtained similar results also for $p=4$, $n=62$, and $\alpha\in\{1.2,1.8\}$; see Figure~\ref{fig:compeig100}.

We end this section by checking how the approximation order of the considered polynomial B-spline collocation method behaves with respect to $p$ for \cmag{smooth solutions} of problem \eqref{eq:FDE}.
More precisely, in Tables \ref{tab:pol33}--\ref{tab:sin} we fix the source function $s(x)$ such that the exact solution of \eqref{eq:FDE} is given by 
\begin{itemize}
	\item $u(x)=x^3(1-x)^3$, and
	\item $u(x)=\sin(\pi x^2)$,
\end{itemize}
respectively. Then, by doubling $n$ repeatedly, we show the infinity-norm of the corresponding errors and the convergence orders for varying $p$ and $\alpha$. The infinity-norm of the error is computed by taking the maximum value of the error sampled in 1024 points uniformly distributed over $[0,1]$. 
\cmag{In the case of standard (non-fractional) diffusion problems, we know that the approximation order for smooth solutions is $p$ for even $p$, and $p-1$ for odd $p$; see \cite{ABHRS}. In the fractional case, we observe a dependency of the approximation order on $\alpha$ that seems to vary as $p+2-\alpha$ for even $p$, and as $p+1-\alpha$ for odd $p$.}

\begin{figure}[t]
	\vspace*{-0.1cm}
	\centering
	\begin{subfigure}[$T_n^{p,\alpha}$]
		{		\includegraphics[width=0.46\linewidth]{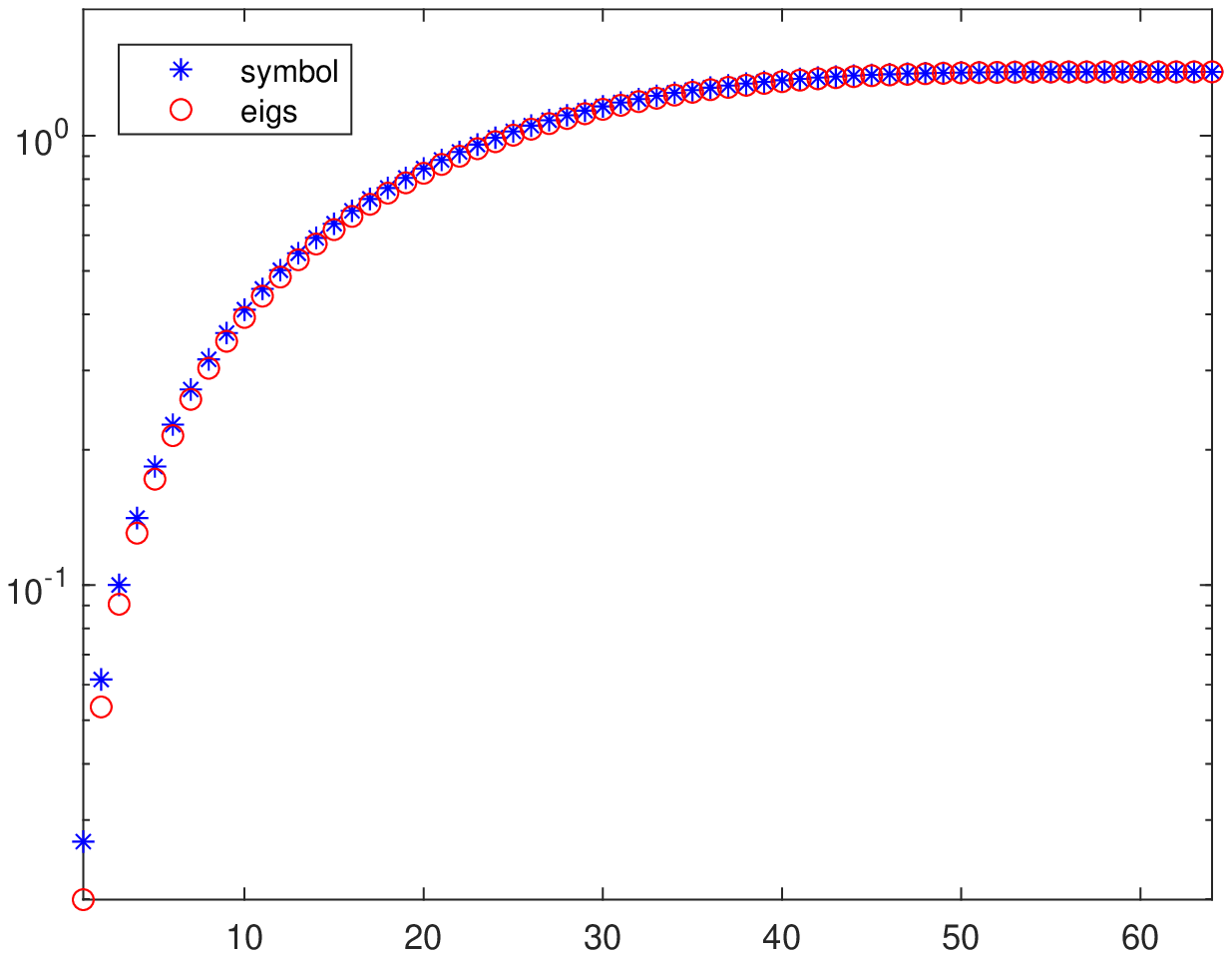}}
	\end{subfigure}
	\begin{subfigure}[$n^{-\alpha}A_{n}^{p,\alpha}$]
		{		\includegraphics[width=0.46\linewidth]{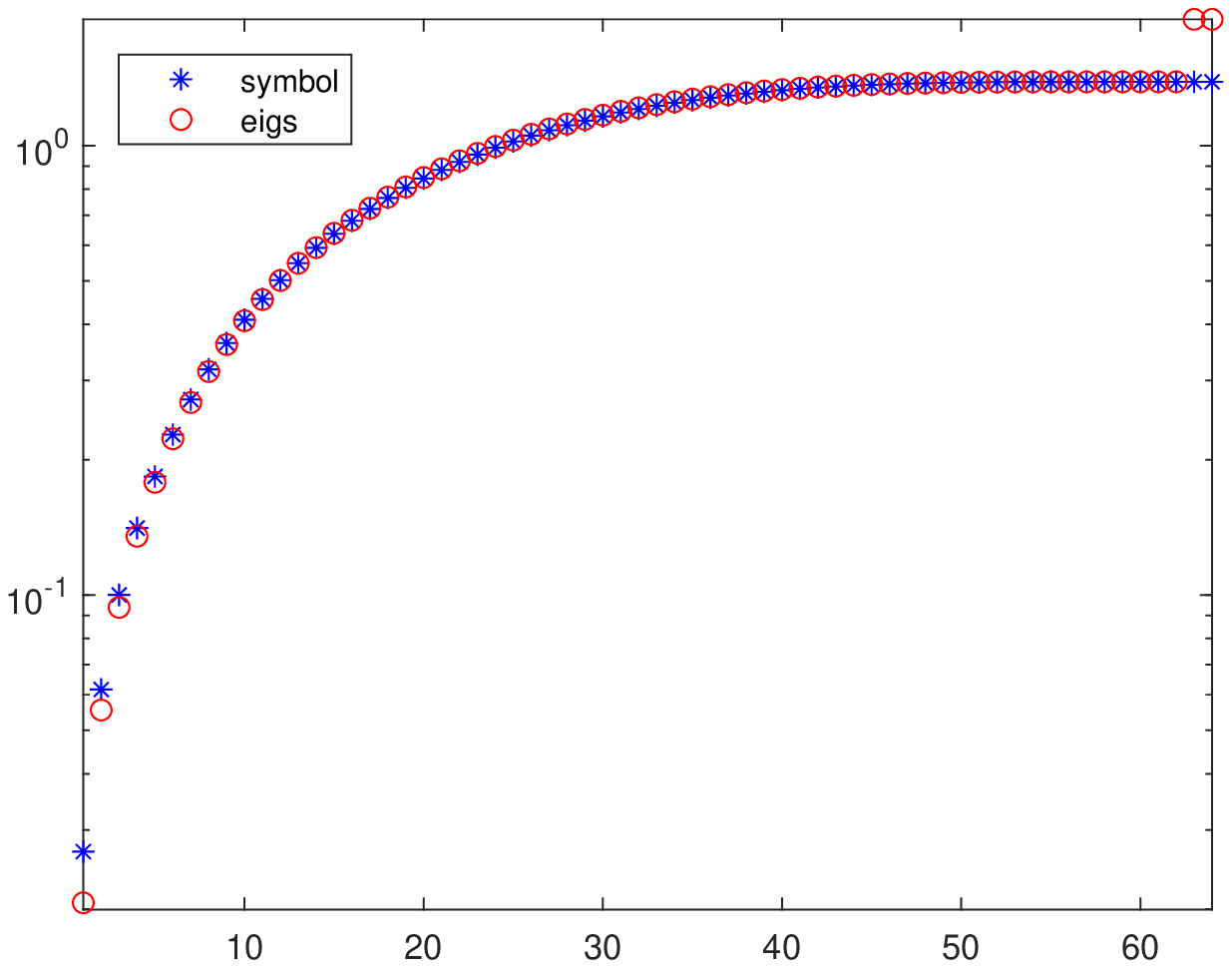}}
	\end{subfigure}
	\\
	\begin{subfigure}[$T_n^{p,\alpha}$]
		{		\includegraphics[width=0.46\linewidth]{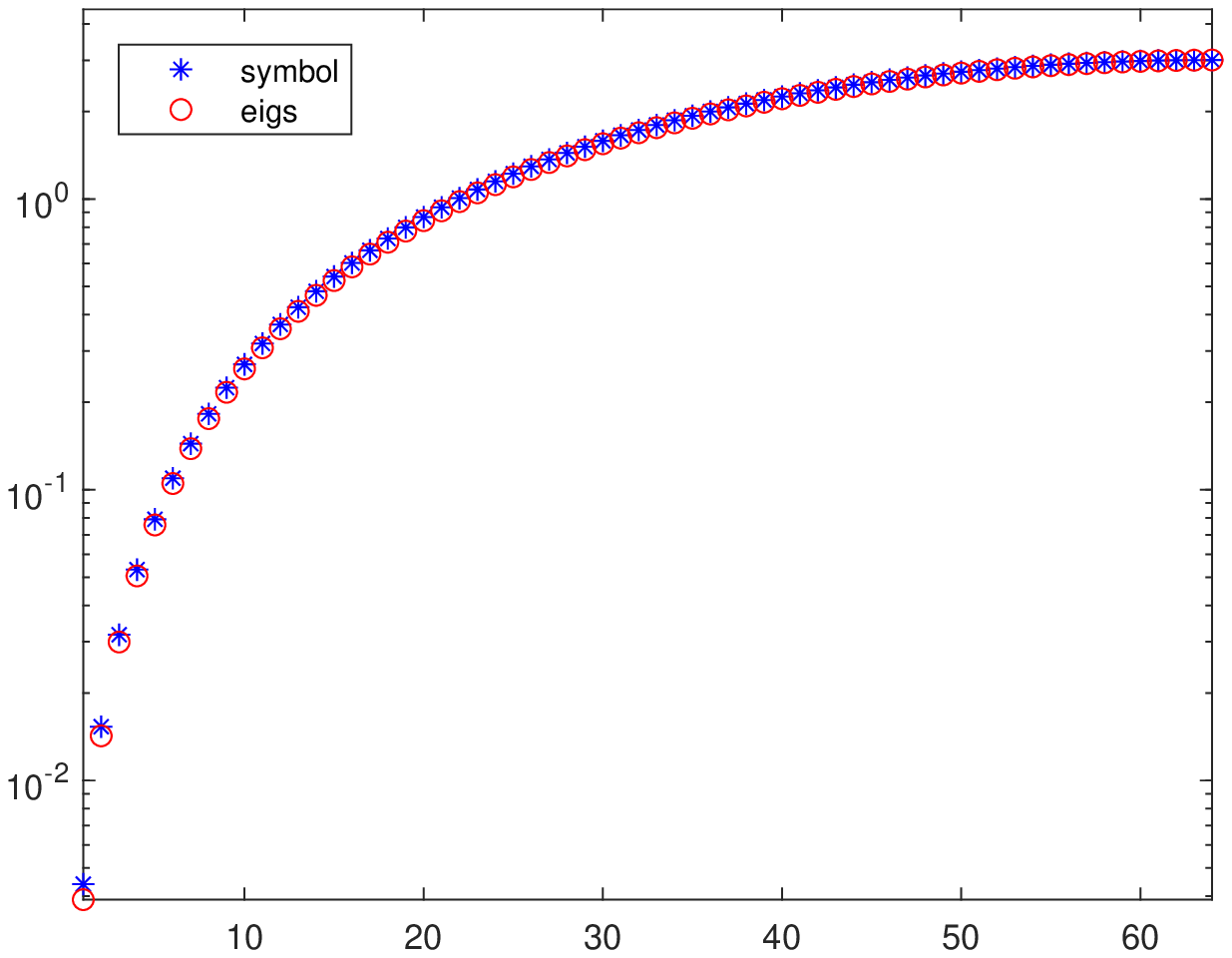}}
	\end{subfigure}
	\begin{subfigure}[$n^{-\alpha}A_{n}^{p,\alpha}$]
		{		\includegraphics[width=0.46\linewidth]{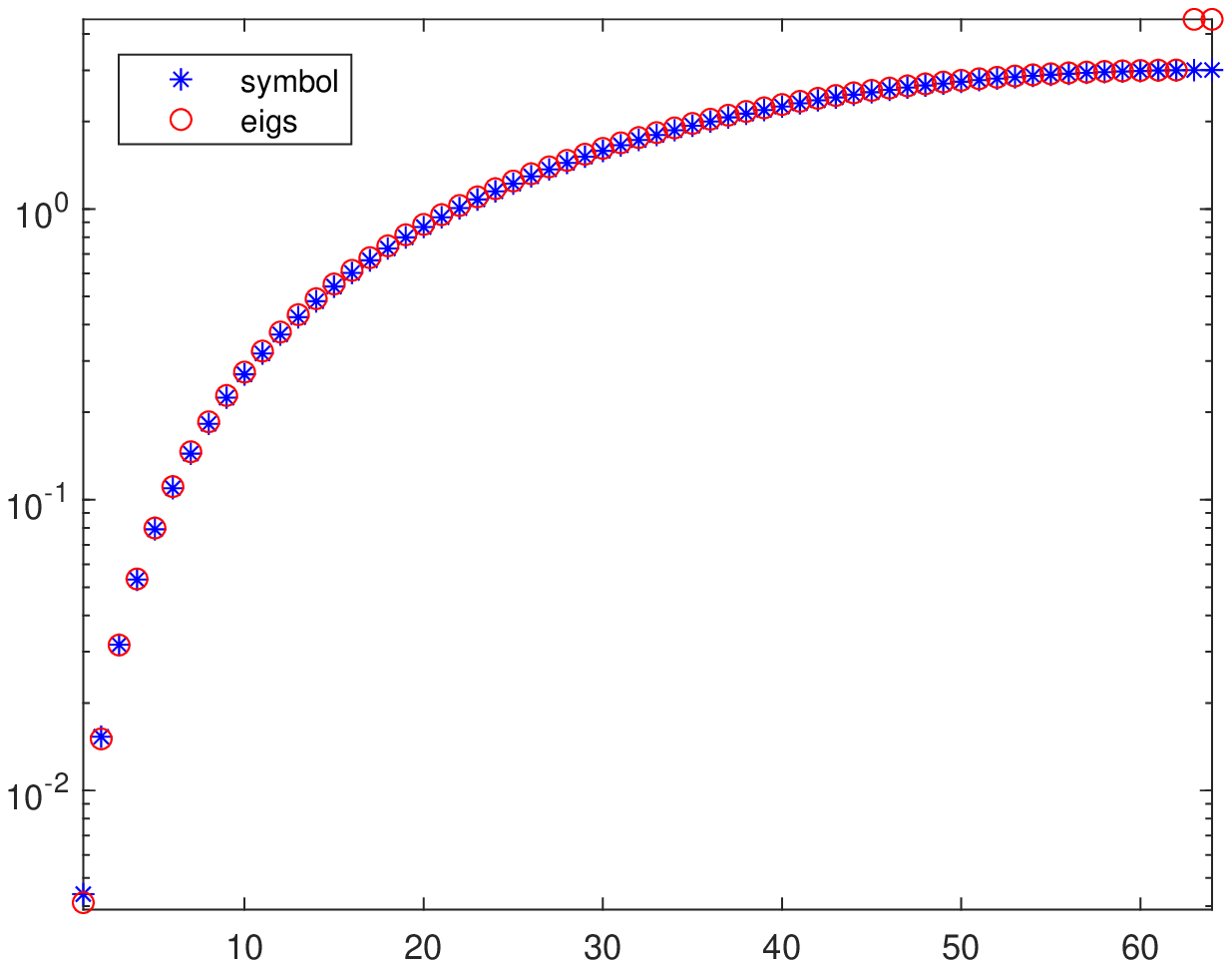}}
	\end{subfigure}
	\caption{Comparison of the eigenvalues of $T_n^{p,\alpha}$ and $n^{-\alpha}A_{n}^{p,\alpha}$ (\textcolor{red}{$\mycircle$}) with a uniform sampling of $f^{p,\alpha}$ on $\varGamma$, ordered in ascending way (\textcolor{blue}{$\ast$}), for $\alpha=1.2$ (top row) and $\alpha=1.8$ (bottom row), $n=63$, $p=3$.} 
	\label{fig:compeig1dot8_100}
\end{figure}
\begin{figure}[t]
	\vspace*{-0.1cm}
	\centering
	\begin{subfigure}[$T_n^{p,\alpha}$]
		{		\includegraphics[width=0.46\linewidth]{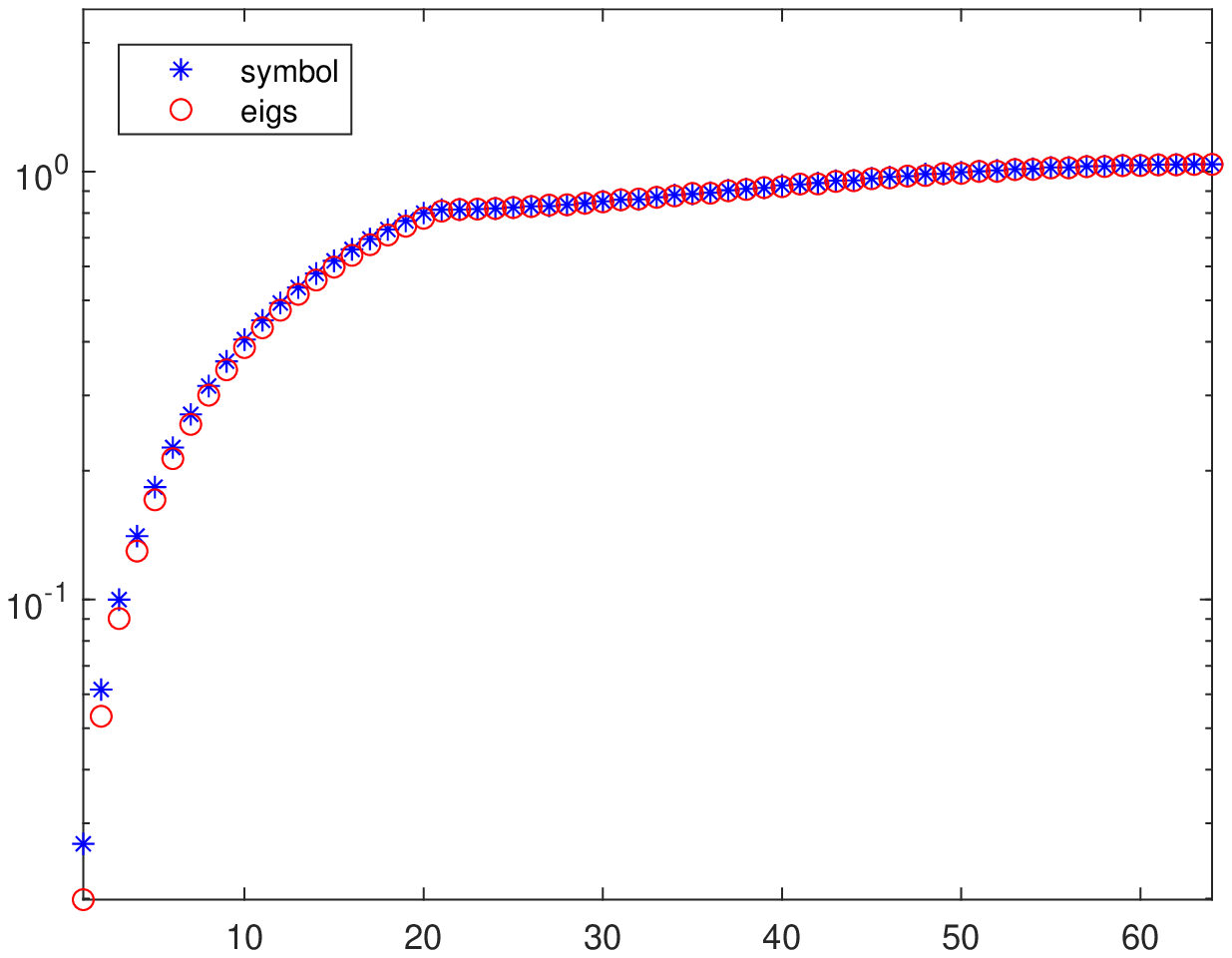}}
	\end{subfigure}
	\begin{subfigure}[$n^{-\alpha}A_{n}^{p,\alpha}$]
		{		\includegraphics[width=0.46\linewidth]{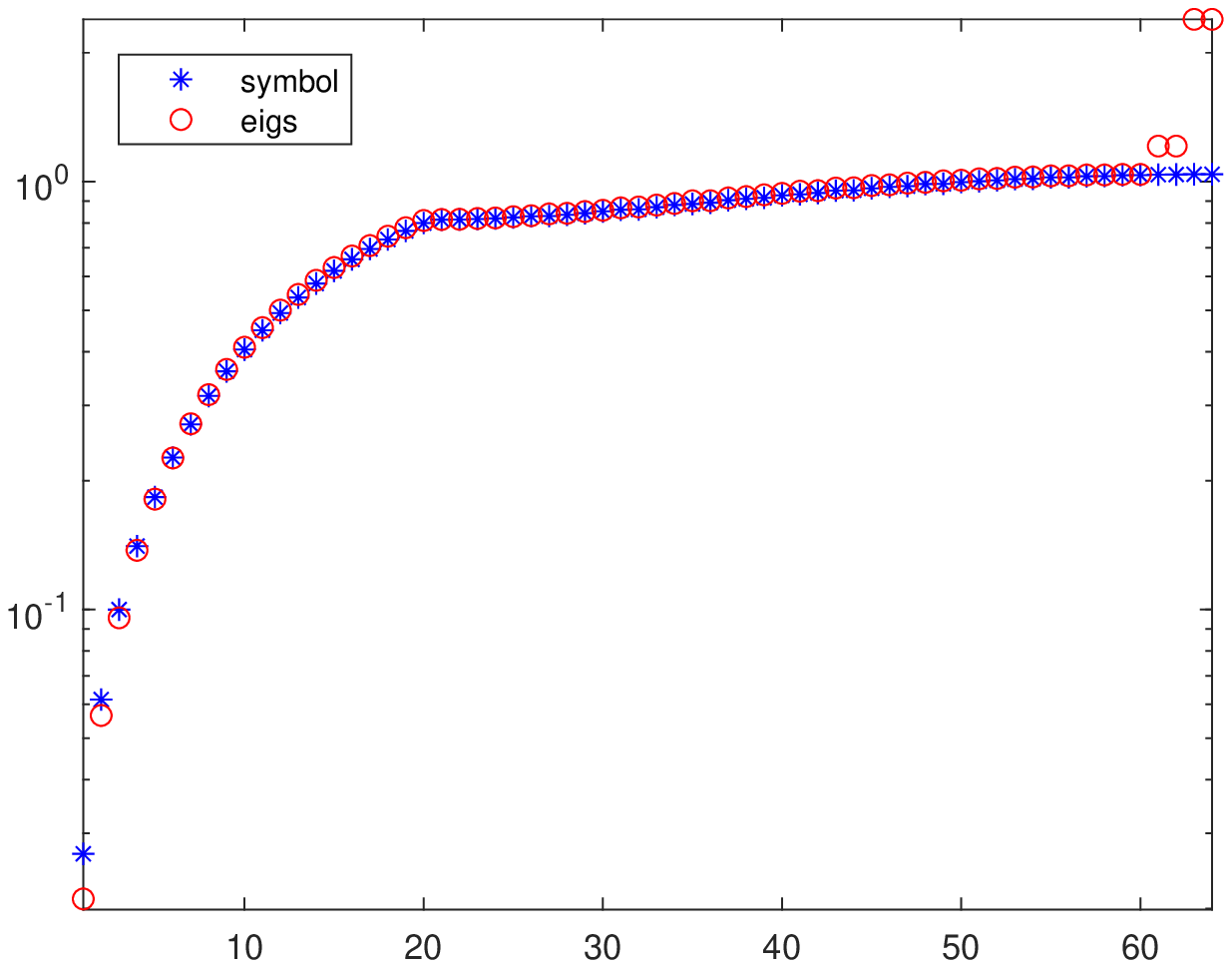}}
	\end{subfigure}
	\\
	\begin{subfigure}[$T_n^{p,\alpha}$]
		{		\includegraphics[width=0.46\linewidth]{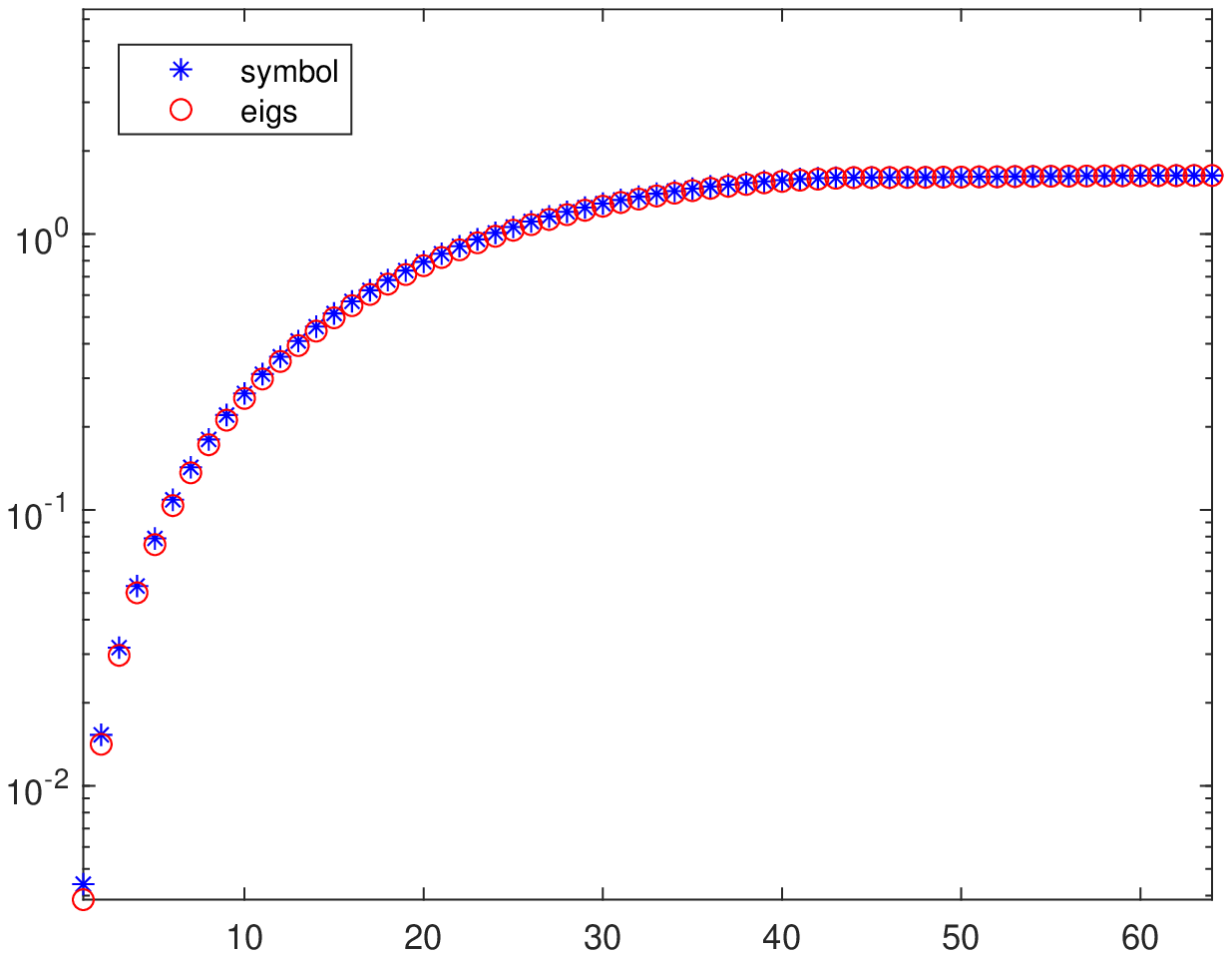}}
	\end{subfigure}
	\begin{subfigure}[$n^{-\alpha}A_{n}^{p,\alpha}$]
		{		\includegraphics[width=0.46\linewidth]{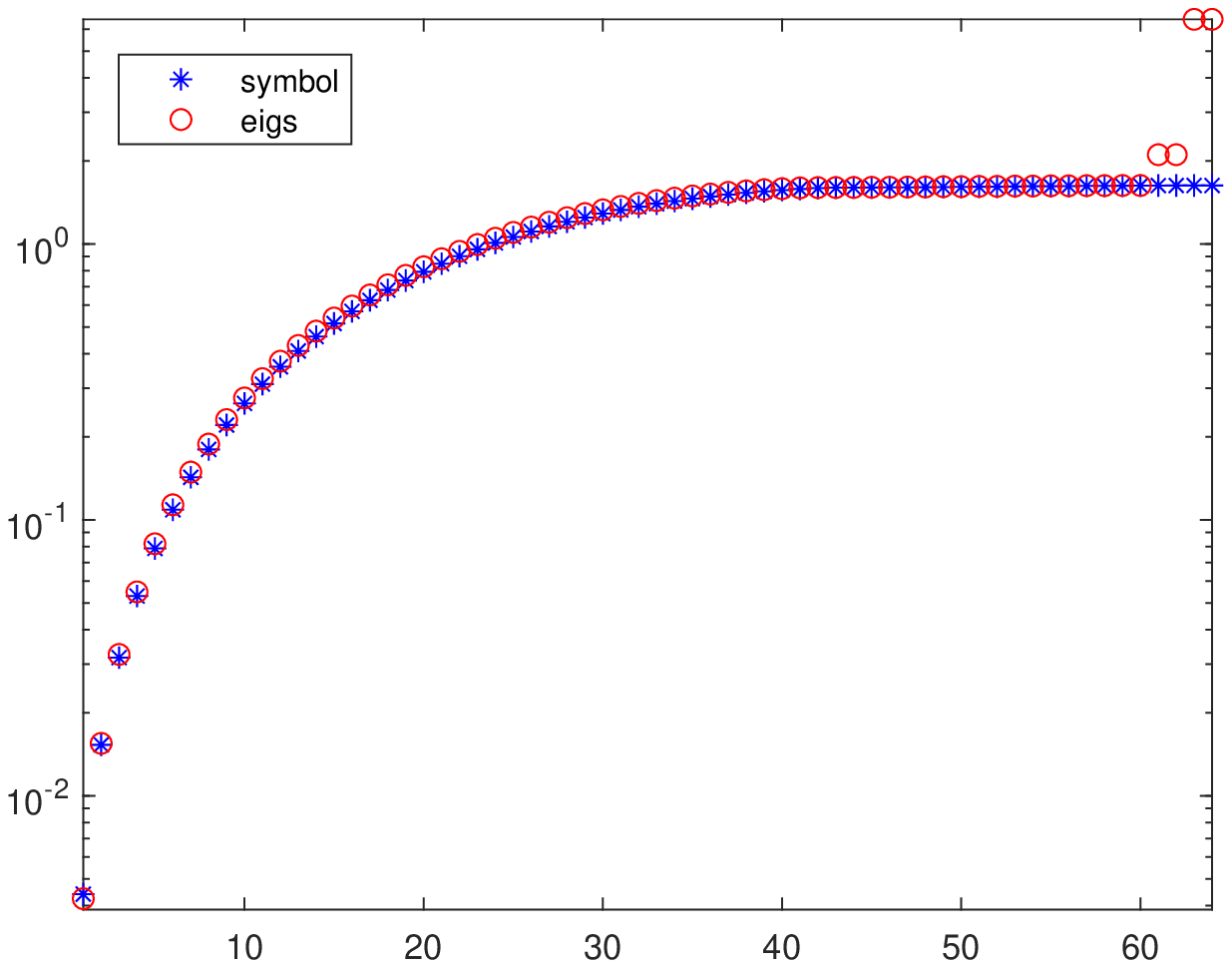}}
	\end{subfigure}
	\caption{Comparison of the eigenvalues of $T_n^{p,\alpha}$ and $n^{-\alpha}A_{n}^{p,\alpha}$ (\textcolor{red}{$\mycircle$}) with a uniform sampling of $f^{p,\alpha}$ on $\varGamma$, ordered in ascending way (\textcolor{blue}{$\ast$}), for $\alpha=1.2$ (top row) and $\alpha=1.8$ (bottom row), $n=62$, $p=4$.} 
	\label{fig:compeig100}
\end{figure}

\begin{table}[htb]
	\scriptsize
	\centering
	\begin{tabular}{|c||c|cc|cc|cc|cc|}
		\hline
		& & & & & & &&&  \\
		\multirow{3}{*}{$\alpha$} & \multirow{3}{*}{$n$}& \multicolumn{2}{|c|}{$p=2$} &\multicolumn{2}{c|}{$p=3$} &\multicolumn{2}{c|}{$p=4$} &\multicolumn{2}{c|}{$p=5$}  \\ \cline{3-10}
		& & & & & & &&&  \\
		& & Error & Order & Error & Order &Error & Order &Error & Order  \\ \hline\hline
		& & & & & & &&&  \\
		\multirow{6}{*}{1.2} 
		& 4 & 1.3146e-03 &  & 1.1197e-03 &  & 2.6802e-04 &  & 4.8556e-05 & \\
		& 8 & 1.5675e-04 & 3.07 & 9.8810e-05 & 3.50 & 1.0317e-05 & 4.70 & 3.4230e-06 & 3.83\\
		& 16 & 2.4941e-05 & 2.65 & 1.5622e-05 & 2.66 & 4.1887e-07 & 4.62 & 1.3853e-07 & 4.63\\
		& 32 & 3.5227e-06 & 2.82 & 2.4433e-06 & 2.68 & 1.6226e-08 & 4.69 & 5.1403e-09 & 4.75\\
		& 64 & 5.0507e-07 & 2.80 & 3.6711e-07 & 2.73 & 5.9986e-10 & 4.76 & 2.1670e-10 & 4.57\\
		& & & $\approx$2.8 &  & $\approx$2.8 & &  $\approx$4.8 & & $\approx$4.8\\
		
		\hline
		& & & & & & &&& \\
		\multirow{5}{*}{1.5} 
		
		& 4 & 1.6170e-03 & & 1.8701e-03 &  & 3.4358e-04 &  & 8.0567e-05 & \\
		& 8 & 1.7117e-04 & 3.24 & 2.0365e-04 & 3.20 & 1.9552e-05 & 4.14 & 7.4745e-06 & 3.43\\
		& 16 & 3.1719e-05 & 2.43 & 2.8530e-05 & 2.84 & 1.0245e-06 & 4.25 & 3.7577e-07 & 4.31\\
		& 32 & 5.8828e-06 & 2.43 & 5.7869e-06 & 2.30 & 4.9498e-08 & 4.37 & 1.7183e-08 & 4.45\\
		& 64 & 1.0458e-06 & 2.49 & 1.0661e-06 & 2.44 & 2.2995e-09 & 4.43 & 8.0703e-10 & 4.41\\
		& & &  $\approx$2.5 &  & $\approx$2.5& &  $\approx$4.5 & & $\approx$4.5\\
		\hline
		& & & & & & &&& \\
		\multirow{5}{*}{1.8} 
		& 4 & 1.9908e-03 &  & 3.1774e-03 &  & 4.3396e-04 & & 1.3425e-04 & \\
		& 8 & 2.5091e-04 & 2.99 & 4.4181e-04 & 2.85 & 3.6073e-05 & 3.59 & 1.5905e-05 & 3.08\\
		& 16 & 4.2953e-05 & 2.55 & 6.8611e-05 & 2.69 & 2.4045e-06 & 3.91 & 9.8386e-07 & 4.01\\
		& 32 & 9.3400e-06 & 2.20 & 1.3336e-05 & 2.36 & 1.4401e-07 & 4.06 & 5.5249e-08 & 4.15\\
		& 64 & 2.0702e-06 & 2.17 & 3.0292e-06 & 2.14 & 8.2251e-09 & 4.13 & 3.0499e-09 & 4.18\\
		& & &  $\approx$2.2 &  &$\approx$2.2 & & $\approx$4.2 & & $\approx$4.2\\
		
		\hline
		\hline
	\end{tabular}
	\caption{Errors and convergence orders of the proposed B-spline collocation method for problem \eqref{eq:FDE} when $u(x)=x^3(1-x)^3$.}\label{tab:pol33}
\end{table}

\begin{table}[htb]
	\scriptsize
	\centering
	\begin{tabular}{|c||c|cc|cc|cc|cc|}
		\hline
		& & & & & & &&&  \\
		\multirow{3}{*}{$\alpha$} & \multirow{3}{*}{$n$}& \multicolumn{2}{|c|}{$p=2$} &\multicolumn{2}{c|}{$p=3$} &\multicolumn{2}{c|}{$p=4$} &\multicolumn{2}{c|}{$p=5$}  \\ \cline{3-10}
		& & & & & & &&&  \\
		& & Error & Order & Error & Order &Error & Order &Error & Order  \\ \hline\hline
		& & & & & & &&&  \\
		\multirow{6}{*}{1.2} 
		& 4 & 4.0099e-02 &  & 1.5948e-02 &  & 6.1393e-03 &  & 1.9341e-03 & \\
		& 8 & 8.4523e-03 & 2.25 & 4.6043e-03 & 1.79 & 2.5317e-04 & 4.60 & 1.0271e-04 & 4.23\\
		& 16 & 1.1497e-03 & 2.88 & 7.8372e-04 & 2.55 & 7.9503e-06 & 4.99 & 2.5175e-06 & 5.35\\
		& 32 & 1.6423e-04 & 2.81 & 1.1786e-04 & 2.73 & 2.5619e-07 & 4.96 & 7.1641e-08 & 5.14\\
		& 64 & 2.3468e-05 & 2.81 & 1.7096e-05 & 2.79 & 9.5594e-09 & 4.74 & 1.0289e-08 & 2.80\\
		& & & $\approx$2.8 &  & $\approx$2.8& &$\approx$ 4.8 & & $\approx$4.8\\
		
		\hline
		& & & & & & &&& \\
		\multirow{5}{*}{1.5} 
		& 4 & 4.2457e-02 &  & 2.4735e-02 &  & 7.7604e-03 &  & 2.6753e-03 & \\
		& 8 & 1.0378e-02 & 2.03 & 7.9809e-03 & 1.63 & 4.3612e-04 & 4.15 & 1.9027e-04 & 3.81\\
		& 16 & 1.7932e-03 & 2.53 & 1.7304e-03 & 2.21 & 1.7374e-05 & 4.65 & 6.3744e-06 & 4.90\\
		& 32 & 3.1466e-04 & 2.51 & 3.1905e-04 & 2.44 & 7.0999e-07 & 4.61 & 2.2599e-07 & 4.82\\
		& 64 & 5.5887e-05 & 2.49 & 5.7202e-05 & 2.48 & 2.9859e-08 & 4.57 & 7.5065e-09 & 4.91\\
		& & & $\approx$2.5 &  &$\approx$2.5 & & $\approx$4.5 & & $\approx$4.5\\
		
		\hline
		& & & & & & &&& \\
		\multirow{5}{*}{1.8} 
		
		& 4 & 4.2801e-02 &  & 3.8129e-02 &  & 9.6792e-03 & & 3.8393e-03 & \\
		& 8 & 1.2259e-02 & 1.80 & 1.4094e-02 & 1.44 & 7.5244e-04 & 3.69 & 3.6381e-04 & 3.40\\
		& 16 & 2.7540e-03 & 2.15 & 3.8466e-03 & 1.87 & 3.9382e-05 & 4.26 & 1.6023e-05 & 4.50\\
		& 32 & 6.0215e-04 & 2.19 & 8.8181e-04 & 2.13 & 2.0021e-06 & 4.30 & 7.1827e-07 & 4.48\\
		& 64 & 1.3172e-04 & 2.19 & 1.9414e-04 & 2.18 & 1.0435e-07 & 4.26 & 3.2796e-08 & 4.45\\
		& & & $\approx$2.2 &  &$\approx$2.2 & & $\approx$4.2 & & $\approx$4.2\\
		
		\hline
		\hline
	\end{tabular}
	\caption{Errors and convergence orders of the proposed B-spline collocation method for problem \eqref{eq:FDE} when $u(x)=\sin(\pi x^2)$.}\label{tab:sin}
\end{table}

\section{Conclusion and future perspective}\label{sec:conclusions}

We focused on a fractional differential equation in Riesz form discretized by a polynomial B-spline collocation method and we showed that, for an arbitrary degree $p$, the resulting coefficient matrices possess a Toeplitz-like structure. We computed the corresponding spectral symbol and we proved that it has a single zero at $0$ of order $\alpha$, with $\alpha$ the fractional derivative order that ranges from $1$ to $2$, and it presents an exponential decay to zero at $\pi$ for increasing $p$ that becomes faster as $\alpha$ approaches $1$. This translates in a mitigated conditioning in the low frequencies and in a deterioration in the high frequencies when compared to second order problems. Moreover, we showed that the behavior of the symbol at $\pi$ is well captured by the symbol corresponding to $\alpha=0$ which is a trigonometric polynomial bounded in the neighborhood of $\pi$. 

As a side result of the symbol computation, we ended up with a new way to express the central entries of the coefficient matrix as inner products of two fractional derivatives of cardinal B-splines.

In addition, we performed a numerical study of the approximation behavior of polynomial B-spline collocation. This study suggests that the approximation order \cmag{for smooth solutions in the fractional case is $p+2-\alpha$ for even $p$, and $p+1-\alpha$ for odd $p$, which is in line with approximation results known for standard (non-fractional) diffusion problems \cite{ABHRS}.}

\cmag{The investigation presented here is intended as a first step towards the use of collocation methods based on high-order polynomial B-splines for FDE problems. 
	In particular, (locally) non-uniform knot sequences could be considered to improve accuracy for non-smooth solutions.
	In this perspective, B-spline collocation methods provide a robust, problem-independent tool to face FDE problems. This robustness makes them an appealing alternative to state-of-the-art methods, such as the elegant collocation/Galerkin spectral methods for approximating the solution of \eqref{eq:FDE} obtained by exploiting the connection  between Jacobi polynomials and pseudo eigenfunctions of the Riesz fractional operator; see \cite{Mao-sinum-18,Mao-sisc-17} and references therein.
	
	The spectral analysis in the present work will provide a strong guidance for forthcoming research. 
	Indeed, the result in Theorem~\ref{thm:spectral-A} is a key ingredient for studying the symbol of matrices arising from B-spline collocation methods for more general FDE problems.
	In particular, additional reaction and advection terms do not modify the symbol of the corresponding matrices; see Remark~\ref{rem:adv-rea}. Furthermore, FDE problems involving non-constant coefficients can be addressed by applying the framework of GLT (Generalized locally Toeplitz) sequences~\cite{GS}. 
	
	Following the results in \cite{cmame2,sinum,mazza2016,sisc}, all the information provided by the symbol can be leveraged for the design of effective preconditioners and fast multigrid/multi-iterative solvers whose convergence speed is independent of the fineness parameters and the approximation parameters as well as of the fractional derivative order; see also Remarks~\ref{rem:mim} and~\ref{rem:mim-A}. 
}

\section*{Acknowledgements}
All authors are members of the INDAM research group GNCS. The first author was partly supported by the GNCS-INDAM Young Researcher Project 2020 titled \lq\lq Numerical methods for image restoration and cultural heritage deterioration''. The last two authors are partially supported by the Beyond Borders Program of the University of Rome Tor Vergata through the project ASTRID (CUP E84I19002250005) and by the MIUR Excellence Department Project awarded to the Department of Mathematics, University of Rome Tor Vergata (CUP E83C18000100006).

\bibliographystyle{amsplain}

\end{document}